\documentclass[10pt]{amsart}

\usepackage{amssymb,latexsym,amsmath,amsthm}
\usepackage[mathscr]{eucal}

\newtheorem{theorem}{Theorem}[section]
\newtheorem{lemma}[theorem]{Lemma}

\theoremstyle{definition}
\newtheorem{remark}[theorem]{Remark}

\newtheorem*{notation}{Notation}
\numberwithin{equation}{section}

\renewcommand{\l}{\lambda}

\newcommand{\RR}{\ensuremath{\mathbb{R}}}

\newcommand{\prtl}{\ensuremath{\partial}}

\newcommand{\supp}{\ensuremath{\mathrm{supp}}}
\newcommand{\veps}{\ensuremath{\varepsilon}}

\newcommand{\tubeh}{\mathcal{T}_{h^{1/2}}(\gamma)}

\newcommand{\e}{\varepsilon}

\title[Refined Kakeya-Nikodym averages in higher dimensions]{Refined and Microlocal Kakeya-Nikodym Bounds of Eigenfunctions in Higher Dimensions}
\thanks{The first author was supported in part by the National Science Foundation grant DMS-1301717, and the second by the National Science Foundation grant DMS-1361476.}

\author[M. D. Blair]{Matthew D. Blair}
\address{Department of Mathematics and Statistics, University of New Mexico, Albuquerque, NM 87131, USA}
\email{blair@math.unm.edu}
\author[C. D. Sogge]{Christopher D. Sogge}
\address{Department of Mathematics, Johns Hopkins University, Baltimore, MD 21093, USA}
\email{sogge@jhu.edu}

\begin{document}
\begin{abstract}
We prove a Kakeya-Nikodym bound on eigenfunctions and quasimodes, which sharpens a result of the authors \cite{blairsoggerefined} and extends it to higher dimensions.  As in the prior work, the key intermediate step is to prove a microlocal version of these estimates, which involves a phase space decomposition of these modes which is essentially invariant under the bicharacteristic/geodesic flow.  In a companion paper \cite{blairsoggetop}, it will be seen that these sharpened estimates yield improved $L^q(M)$ bounds on eigenfunctions in the presence of nonpositive curvature when $2 < q < \frac{2(d+1)}{d-1}$.
\end{abstract}
\maketitle

\section{Introduction and main results}
Let $(M,g)$ be a smooth, compact, connected, boundaryless Riemannian manifold of dimension $d \geq 2$.  Let $\Delta_g$ be the nonpositive Laplace-Beltrami operator which is self adjoint with respect to the Riemannian measure $dV_g$.  The compactness assumption ensures that the spectrum of $\Delta_g$ is discrete, so that there exists an orthonormal basis for $L^2(M)$ consisting of eigenfunctions $e_{\l}$ satisfying
\begin{equation}\label{classicalefcn}
-\Delta_g e_{\l} = \l^2 e_{\l}, \qquad\text{ or equivalently, }\qquad \sqrt{-\Delta_g} e_{\l} = \l e_{\l} .
\end{equation}
Given such an eigenfunction with $\lambda >0$, we set $h=1/\lambda$, obtaining a solution to the semiclassical problem $(h^2\Delta_g+1)e_{h^{-1}}=0$.

In this work, we are interested in $L^q$ bounds on eigenfunctions and quasimodes associated to the semiclassical operator $h^2\Delta_g+1$.  A family of functions $\{\psi_h\}$ either defined for $h$ in some subinterval $h\in(0,h_0]\subset (0,1]$ or a decreasing sequence in this interval tending to 0 will be considered to be an admissible family of quasimodes if
\begin{equation}\label{admissquasi}
\begin{gathered}
\|\psi_h\|_{L^2(M)}=1,\\
\|(h^2\Delta_g+1)^k\psi_h\|_{L^2(M)} = O(h^k), \qquad k=1, \dots, \left\lceil \frac{d}{d+1}+\frac{d-1}{2} \right\rceil.
\end{gathered}
\end{equation}
Eigenfunctions normalized in $L^2$ define admissible quasimodes, as are normalized functions in the range of $\mathbf{1}_{[1,1+h]}(h\sqrt{-\Delta_g})$.  In the classical notation \eqref{classicalefcn}, it is easily verified that $L^2$-normalized functions in the ranges of
\begin{equation}\label{classicalquasi}
\mathbf{1}_{[\lambda,\lambda+1/T]} (\sqrt{-\Delta_g}) \qquad \text{ and } \qquad \rho(T(\lambda-\sqrt{-\Delta_g}))
\end{equation}
define admissible quasimodes for $T =T(\lambda)\geq 1$ a nondecreasing function, and in the latter case $\rho\in \mathcal{S}(\RR)$.

In this work, the interest is in upper bounds on the norm of these eigenfunctions and quasimodes in $L^q(M)$ for $2<q <\frac{2(d+1)}{d-1}$.  In \cite{sogge88}, the second author proved that $\|\psi_h\|_{L^q(M)} = O(h^{-\delta(q)})$ where $\delta(q) = \max(\frac{d-1}{2}(\frac 12-\frac 1q), \frac{d-1}{2}-\frac dq)$ as $h \to 0$ (or equivalently $\lambda \to \infty$).  These bounds are universal in that they apply to any $(M,g)$ as above and any admissible quasimode as in \eqref{admissquasi}.  Moreover, if one considers the large classes of quasimodes given by functions in the range of $\mathbf{1}_{[1,1+h]}(h\sqrt{-\Delta_g})$ (or in the classical notation, $\mathbf{1}_{[\lambda,\lambda+1]} (\sqrt{-\Delta_g})$), it can be seen that this exponent cannot be improved.  Indeed, when $2<q\leq \frac{2(d+1)}{d-1}$ one can find functions $\psi_h$ in the range of these spectral projectors for which $|\psi_h(x)|$ is roughly constant in a tubular neighborhood of radius $\approx h^{1/2}$ about a geodesic segment when $2<q\leq \frac{2(d+1)}{d-1}$ and rapidly decreasing outside of this set so that $
\|\psi_h\|_{L^q(M)}\approx  h^{-\frac{d-1}{2}(\frac 12-\frac 1q)}.
$ When $\frac{2(d+1)}{d-1}\leq q \leq \infty$, there exist modes weakly decaying outside of a ball of radius $h$ which similarly saturate these bounds.

On the round sphere $\mathbb{S}^d$, it is well known that the spectrum of $\sqrt{-\Delta_{\mathbb{S}^d}}$ is of the form $\sqrt{k(k+d-1)} = k+O(\frac 1k)$, $k\in \mathbb{N}\cup\{0\}$.  Consequently, these spectral projectors in the previous paragraph are nearly the same as projections onto an eigenspace, and the $L^q(M)$ bounds are saturated by families of exact eigenfunctions, in particular the highest weight spherical harmonics when $2<q\leq \frac{2(d+1)}{d-1}$ and the zonal harmonics when $\frac{2(d+1)}{d-1}\leq q \leq \infty$.  However, the spectrum of $\sqrt{-\Delta_{\mathbb{S}^d}}$ and the dynamics of the geodesic flow on $\mathbb{S}^d$ are very unique and these universal bounds are not expected to be optimal for eigenfunctions in most other geometries (save those with elliptic closed geodesics or a full measure of periodic geodesics in the sense of \cite{soggezelduke}).  A problem of great interest is to determine geometric conditions which show that these universal bounds can be improved for eigenfunctions or even quasimodes of shrinking width with $T(\lambda)\nearrow \infty$ in \eqref{classicalquasi}.

This work develops a condition for obtaining improvements on the universal bounds when $2 <q < \frac{2(d+1)}{d-1}$.  Let $\varPi$ denote the space of unit length geodesic segments.  Given $\gamma \in \varPi$, define $\mathcal{T}_{\veps}(\gamma)$ to be its $\veps$-neighborhood
\begin{equation}\label{tubedef}
\mathcal{T}_{\veps}(\gamma)=\{x\in M: d_g(x,\gamma)< \veps\}.
\end{equation}
For $h>0$ we then define the Kakeya-Nikodym norm of a quasimode $\psi_h$ as
\begin{equation}\label{KNnorm}
\| \psi_h \|_{KN}^2 = \sup_{\gamma \in \varPi} h^{-\frac{d-1}{2}} \int_{\tubeh} |\psi_h|^2\,dV_g.
\end{equation}
It is also convenient to use the analogous norm without (approximte) averages
\begin{equation}\label{KNnormNonavg}
||| \psi_h |||_{KN}^2 = \sup_{\gamma \in \varPi} \int_{\tubeh} |\psi_h|^2\,dV_g.
\end{equation}
The norm thus depends on the frequency scale $h$ of the quasimode under consideration.  Note that if $\psi_h$ is rapidly decreasing outside a tube $\tubeh$, as in the case of the highest weight spherical harmonics, then $||| \psi_h |||_{KN} \approx 1$.  Otherwise, one expects better bounds to be satisfied.

The following is the primary result of this work.
\begin{theorem}\label{thm:mainthm}
Let $\psi_h$ be a member of an admissible family of quasimodes as in \eqref{admissquasi}.  If $d \geq 2$ and $\frac{2(d+2)}{d} <q <\frac{2(d+1)}{d-1}$, then
\begin{equation}\label{higherdimboundnavg}
\|\psi_h\|_{L^q(M)} \lesssim
h^{-\frac{d-1}{4}+\frac{d-1}{2q}} |||\psi_h|||_{KN}^{\frac{2(d+1)}{q(d-1)}-1} = h^{-\delta(q)}|||\psi_h|||_{KN}^{\frac{2(d+1)}{q(d-1)}-1}.
\end{equation}
Note that when $d=2$, this means that for $4 < q < 6$,
\begin{equation}\label{twodimboundnavg}
\|\psi_h\|_{L^q(M)} \lesssim h^{-\frac14 +\frac{1}{2q}}|||\psi_h|||_{KN}^{\frac 6q -1} = h^{-\delta(q)}|||\psi_h|||_{KN}^{\frac 6q -1} .
\end{equation}
Moreover, when $d=2$ and $q=4$, we also have that
\begin{equation}\label{twodimboundL4navg}
\|\psi_h\|_{L^4(M)} \lesssim h^{-\frac 18} \sqrt{|\log h|}\;|||\psi_h|||_{KN}^{\frac 12} = h^{-\delta(4)} \sqrt{|\log h|}\;|||\psi_h|||_{KN}^{\frac 12} .
\end{equation}
\end{theorem}
Here $\delta(q) = \frac{d-1}{2}(\frac 12-\frac 1q)$ as in the classical bounds of the second author \cite{sogge88}.  The theorem thus furnishes a sufficient condition for improving on these classical bounds, as it means that improvements on the trivial $|||\psi_h|||_{KN} = O(1)$ bounds translates to improvements on the $L^q(M)$ bounds (see our discussion of \cite{blairsoggetop} below).

Expressing our main theorem in this manner allows for a good discussion of its consequences, which we do below and in our companion paper \cite{blairsoggetop}.  However, in proving the theorem, it is convenient to instead work with the averaged norms \eqref{KNnorm} and show following bounds which are equivalent to \eqref{higherdimboundnavg}, \eqref{twodimboundL4navg}:
\begin{equation}\label{higherdimbound}
\|\psi_h\|_{L^q(M)} \lesssim
h^{-\frac{d-1}{2}+\frac dq} \|\psi_h\|_{KN}^{\frac{2(d+1)}{q(d-1)}-1}, \qquad d \geq 2 ,
\end{equation}
\begin{equation}\label{twodimboundL4}
\|\psi_h\|_{L^4(M)} \lesssim \sqrt{|\log h|}\;\|\psi_h\|_{KN}^{\frac 12}, \qquad d = 2.
\end{equation}

In \cite{blairsoggerefined}, the authors proved the bound \eqref{twodimboundL4}, but with the factor $\sqrt{|\log h|}$ on the right replaced by $h^{-\veps}$, where $\veps>0$ can be taken to be arbitrarily small.  This is because that work actually proves lossless estimates involving averages of the mass over slightly thicker tubular neighborhoods $h^{1/2-\veps}$.  The present work thus sharpens this result and extends it to higher dimensions.  Estimates of this type extend results of Bourgain \cite{bourgain} and the second author \cite{soggekaknik} in two dimensions, and the authors \cite{blairsogge} in higher dimensions.  This latter set of results show that the universal $L^q$ bounds on eigenfunctions due to the second author \cite{sogge88} can be improved to $o(h^{-\frac{d-1}{2}(\frac 12-\frac 1q)})$ when $|||\psi_h|||_{KN}=o(1)$ and $2< q < \frac{2(d+1)}{d-1}$.

Obtaining lossless estimates involving tubes of diameter $h^{1/2}$ is significant for applications.  In a companion paper \cite{blairsoggetop}, it will be shown that
\[
|||\psi_h|||_{KN}=O(|\log h|^{-\sigma}), \qquad \text{for some $\sigma=\sigma(d) >0$},
 \]
for quasimodes in the range of
\begin{equation}
\rho(T(\lambda-\sqrt{-\Delta_g})) = \rho(Th^{-1}(1-h\sqrt{-\Delta_g})), \qquad h= 1/\lambda,
\end{equation}
as in \eqref{classicalquasi} with $T\approx \log\lambda = |\log h|$.  The small $h^{1/2}=\lambda^{-1/2}$ thickness of the tubes turns out to be significant in the proof of such bounds.  Combining these with Theorem \ref{thm:mainthm} thus yields a logarithmic gain in the $L^q$ bounds when $2 < q < \frac{2(d+1)}{d-1}$ for quasimodes in the range of
\begin{equation}
\mathbf{1}_{[\lambda,\lambda+(\log \lambda)^{-1}]}(\sqrt{-\Delta_g}) =
\mathbf{1}_{[1,1+h|\log h|^{-1}]}(h\sqrt{-\Delta_g}) .
\end{equation}
Namely, for such quasimodes we have
\[
\|\psi_h\|_{L^q(M)} \lesssim
h^{-\frac{d-1}{4}+\frac{d-1}{2q}} |\log h|^{-\sigma(\frac{2(d+1)}{q(d-1)}-1)}.\]
Moreover, the second author has shown \cite{soggeforthcoming} that averaging over these smaller tubes allows one to see that upper bounds on the $L^1(M)$ norm of allows one to detect scarring phenomena for eigenfunctions and quasimodes.  More precisely, it is observed that using H\"older's inequality,
\begin{equation*}
 |||\psi_h|||_{KN}^{-\frac{2(d+1)-(d-1)q}{(d-1)(q-2)}} \lesssim h^{-\frac{d-1}{4}}\|\psi_h\|_{L^1(M)} \qquad \mbox{for $\frac{2(d+2)}{d}<q<\frac{2(d+1)}{d-1}$}.
\end{equation*}
Thus if the right hand side is $\approx 1$, it can be seen that the mass of $\psi_h$ must concentrate in an $h^{1/2}$ tubular neighborhood about a geodesic segment.  Indeed, if the mass of $\psi_h$ were mostly concentrated in a $\tubeh$, then it is not hard to see that $\|\psi_h\|_{L^1(M)} \lesssim h^{\frac{d-1}{4}}$.  The main idea in \cite{soggeforthcoming} is to show these upper bounds on the $L^1(M)$ norm imply such concentration must occur.

The present work does not address how to improve upon the universal $L^q(M)$ bounds of the second author in the range $\frac{2(d+1)}{d-1} < q \leq \infty$ or when $q=\frac{2(d+1)}{d-1}$.  Improvements in the former range, based on the measure of closed loops in the cosphere spaces $S^*_xM$, appear in \cite{soggetothzelditch}, \cite{soggezelduke}, \cite{soggezelfocal} and stronger improvements for nonpositive curvature were shown in \cite{Berard}, \cite{HT}.  As alluded to above, any condition for improvements on the universal $L^q$ bounds in these cases must address the existence of modes which decay weakly outside of balls of radius $h$, and the norms in \eqref{KNnorm} above are insufficient for this purpose.  The results in \cite{HR} obtain a logarithmic gain in the universal estimates in the presence of negative curvature, for all $2 < q \leq \infty$ via an estimate at $q=\frac{2(d+1)}{d-1}$ and interpolation, but the methods rely on generalized Weyl laws, and hold only for a full density subsequence of eigenfunctions.

The discussion above shows that the exponent in the two dimensional bound \eqref{twodimboundnavg} is perhaps not so surprising as it would be the outcome of interpolating between the $L^4(M)$ bound (without the logarithmic factor) and the universal $L^6(M)$ bound of the second author.  The same idea holds for the higher dimensional bound \eqref{higherdimbound}: if this were to hold when $q=\frac{2d}{d-1}$ (so that the exponent of $h$ vanishes), then the estimates appearing here would be the result of interpolating between this bound and the classical bounds of the second author at $q=\frac{2(d+1)}{d-1}$.  However, the validity of  \eqref{higherdimbound} when $\frac{2d}{d-1} \leq q \leq \frac{2(d+2)}{d}$ and $d\geq 3$ could be subtle.  Indeed, when $d\geq 3$, the present work relies on bilinear estimates of Lee \cite{leebilinear} arising from the bilinear approach to the Fourier restriction problem that originated in works of Tao, Vargas, and Vega \cite{tvv}, Wolff \cite{wolff}, and Tao \cite{taobilinear}.  Explicit examples show that such bilinear estimates fail to hold in the regions $\frac{2d}{d-1} \leq q \leq \frac{2(d+2)}{d}$ (see \cite{tvv}).

\begin{remark}
The present work uses the semiclassical formalism of Koch, Tataru, and Zworski in \cite{ktzsemi}.  It is thus likely that Theorem \ref{thm:mainthm} can be generalized to semiclassical pseudodifferential operators $P(x,hD)$ with symbols $p(x,\xi)$ whose characteristic sets $\{\xi \in T^*_xM :p(x,\xi)=0\}$ have nonvanishing second fundamental form.  This is provided sets analogous to the $\tubeh$ above can be defined in a manner that allows for a substitute for Lemma \ref{lem:MKNandKN} below.  While we do not examine this issue in detail, our methods do easily treat semiclassical operators of the form $h^2\Delta_g+V$, $V \in C^{\infty}$ with $0< \inf_{x\in M} V(x)$, by passing to the ``Jacobi metric" $\tilde{g}(x):=V(x)g(x)$ (see \cite[p.228]{abrahammarsden}).  In this case, if one has an admissible quasimode defined by replacing $(h^2\Delta_g+1)\psi_h$ by $(h^2\Delta_g+V)\psi_h$ in \eqref{admissquasi}, then $\psi_h$ is a quasimode with respect to $h^2\Delta_{\tilde{g}}+1$ in that
\[
\left\|(h^2\Delta_{\tilde{g}}+1)^k \psi_h\right\|_{L^2(M)} = O(h^k), \text{ for }
k=1, \dots, \left\lceil \frac{d}{d+1}+\frac{d-1}{2} \right\rceil.
\]
Indeed, it is verified that the difference $V\cdot \Delta_{\tilde{g}} - \Delta_{g}$ is a first order differential operator on $M$, which implies that
Theorem \ref{thm:mainthm} applies in this setting, defining the sets $\tubeh$ with respect to $\tilde{g}$.
\end{remark}
\subsection*{Outline of the paper} In \S\ref{sec:semiclassical}, we review aspects of the semiclassical approach to $L^q$ bounds for eigenfunctions introduced in \cite{ktzsemi}.  We also prove an Egorov theorem for symbols in the critical ``$S_{1/2}(1)$" class that is needed in our work.  Section \ref{sec:MKNoutline} provides the main outline of the proof of Theorem \ref{thm:mainthm}, considering ``microlocal" Kakeya-Nikodym norms and reducing matters to showing certain almost orthogonality bounds and bilinear estimates in $L^{\frac q2}$.  The last two sections, \S\ref{sec:ao} and \S\ref{sec:bilinear}, treat these almost orthogonality bounds and bilinear estimates respectively.
\subsection*{Acknowlegement}  The authors are grateful to Steve Zelditch for several helpful conversations and to the anonymous referee for numerous thoughtful suggestions which improved the exposition in this work.

\section{Semiclassical analysis}\label{sec:semiclassical}
\subsection{Preliminary reductions}
Since $\|(h^2\Delta_g)^k \psi_h\|_{L^2(M)} =O(1)$ for the $k$ in \eqref{admissquasi}, elliptic regularity shows that $\|\psi_h\|_{H^{2k}_h(M)} =O(1)$ in the semiclassical Sobolev spaces.  Hence commutator estimates show that for a smooth bump function $\varphi$,
\[
\|(h^2\Delta_g + 1)^k (\varphi \psi_h)\|_{L^2(M)} = O(h^k) \text{ for }
k=1, \dots, \left\lceil \frac{d}{d+1}+\frac{d-1}{2} \right\rceil.
\]
Hence by multiplying $\psi_h$ by a bump function within a finite partition of unity, we may assume it is supported in a suitable coordinate chart given by a cube of sidelength $2\e$ centered at the origin in $\RR^d$, and in particular
\begin{equation}\label{epsdef}
\supp(\psi_h) \subset \left[-\frac{\e}{2},\frac{\e}{2}\right]^{d},
\end{equation}
for some $\e>0$ sufficiently small.  Moreover, we may assume that $g^{ij}(0) = \delta^{ij}$.

Next, we let $\chi$ be a suitable smooth cutoff to a neighborhood of the unit sphere such that for $k$ as in \eqref{admissquasi}
\begin{equation}\label{gainawaychar}
\|(1-\chi(hD))\psi_h\|_{L^2(M)} \lesssim \|(h^2\Delta_g +1)^k\psi_h \|_{L^2(M)} +h^k \|\psi_h\|_{L^2(M)}\lesssim h^k.
\end{equation}
Indeed, for a suitable choice of $\chi$, $h^2\Delta_g +1$ is elliptic on $\supp(1-\chi)$, so elliptic regularity bounds yield the first inequality and quasimode properties the second.  Taking $k=\left\lceil \frac{d}{d+1}+\frac{d-1}{2} \right\rceil$, semiclassical Sobolev embedding gives for $ 2< q <\frac{2(d+1)}{d-1}$
\begin{align*}
  \|(&1-\chi(hD))\psi_h\|_{L^q(M)}
  \lesssim h^{-(\frac d2-\frac dq)}\sum_{|\alpha|\leq 1}\|(hD)^\alpha (1-\chi(hD))\psi_h \|_{L^2(M)}\\
  &\lesssim h^{-\frac{d}{d+1}}\left(\|(h^2\Delta_g + 1)^k (1-\chi(hD))\psi_h \|_{L^2(M)} + \|(1-\chi(hD))\psi_h \|_{L^2(M)}\right)\\
  &\lesssim h^{k-\frac{d}{d+1}} \leq h^{\frac{d-1}{2}},
\end{align*}
where the second inequality follows from elliptic regularity (as in \cite[Lemma 2.6]{ktzsemi}) and the last follows from the observation \eqref{gainawaychar} above (for a slightly different choice of $\chi$) as the symbol of $[(h^2\Delta_g + 1)^k ,(1-\chi(hD))]$ is supported away from the unit sphere, at least up to $O(h^k)$ error.  Since we may cover $M$ with $O(h^{-\frac{d-1}{2}})$ tubular neighborhoods $\tubeh$, $\|(1-\chi(hD))\psi_h\|_{L^q(M)}$ is in turn bounded by the right hand side of the inequalities in Theorem \ref{thm:mainthm}.

Taking a further partition of unity in the frequency variable, we may alter $\chi(\zeta) $ so that it is supported in a conic set of the form
\begin{equation}\label{conesupp}
\{\zeta: -\zeta_1 \geq C|\zeta'|, |\zeta| \approx 1\}, \qquad \zeta'=(\zeta_2, \dots, \zeta_n),
\end{equation}
for some $C>0$ sufficiently large.  The main idea in \cite{ktzsemi} is that after this further microlocalization, we may write the characteristic surface $\sum_{ij}g^{ij}(z)\zeta_i\zeta_j = 1$ as a graph in the $\zeta'$ variables, which leads us to a evolution equation in the first variable.  Indeed, over \eqref{conesupp}, the principal symbol of $h^2\Delta_g +1$ can be written in the form
\begin{equation}\label{graph}
-\sum_{i,j=1}^d g^{ij}(z)\zeta_i\zeta_j +1 = -a(z,\zeta)(\zeta_1 + p(z,\zeta')), \quad (z,\zeta) \in \supp(\psi_h)\times \supp(\chi),
\end{equation}
with $a(z,\zeta) , p(z,\zeta') > 0$, by taking $\e>0$ sufficiently small in \eqref{epsdef} and the aperture of the cone sufficiently small in \eqref{conesupp}.  Moreover taking $\veps$ small above, we may assume that the restricted matrix $\{g^{ij}(z)\}_{i,j=2}^{d}$ is positive definite when $z \in [-\veps,\veps]^d$.

We now have that
\begin{align*}
(hD_{z_1}+P(z,hD_{z'}))(\chi(hD)\psi_h) = h f_h,\\
f_h := \frac{1}{h}b(z,hD)[h^2\Delta_g,\chi(hD)]\psi_h,
\end{align*}
for some symbol $b$ compactly supported in all variables.  In particular, we may assume that $\supp(b(\cdot,\zeta)) \subset [-\veps,\veps]^d$ for every $\zeta$.  Moreover, we have that
\begin{equation}\label{fhbound}
\|f_h\|_{L^2(\RR^d)} \lesssim \|\psi_h\|_{L^2(\RR^d)}
\end{equation}
and up to an error term which is $O(h^\infty)$ as an operator on $L^2$, the composition
\begin{equation}\label{fhoperator}
b(z,hD)[h^2\Delta_g,\chi(hD)]
\end{equation}
is an operator with a compactly supported symbol.

\begin{notation}
Since the operator with symbol $(\zeta_1 + p(z,\zeta'))$ is naturally an evolution equation, it is now convenient to take a different convention on the notation.  We will use $r,s,t$ to denote variables in $\RR$ and $w,x,y,\xi,\eta$ to denote those in $\RR^{d-1}$ so that expressions such as $(t,x)$ give variables in $\RR^d$. The variable $z$ will still denote those in $\RR^d$ on the few occasions this is needed. The relation for $\psi_h$ thus reads
\[
(hD_{t}+P(t,x,hD))\psi_h = h f_h
\]
where it is understood that $P(t,x,hD)$ (or $P(t)$ when the dependence on $x,h$ is not crucial) is a family of semiclassical pseudodifferential operators acting on the $x$ variables.  Keeping in line with these conventions, the symbols we consider in the remainder of this work will be quantized in the ``$x$" variables over $\RR^{d-1}$ and never in the first variable.  Note that given the restriction of $\chi$ above to a cone of sufficiently small aperture, we may also assume that for some smooth bump function $\tilde \chi$ with $\supp(\tilde\chi) \subset \{|\xi| \ll 1\}$ (so that the quantization is only in the $x$ variables)
\begin{equation}\label{frequencyerror}
 \|(1-\tilde\chi)(hD)\psi_h(t,\cdot)\|_{L^2} , \|(1-\tilde\chi)(hD)f_h(t,\cdot)\|_{L^2} = O(h^\infty), \quad |t |\leq \veps.
\end{equation}
\end{notation}

Let $\tilde{\psi}_h := \chi(hD)\psi_h$ and let $S(t,s)$ be the family of unitary operators satisfying
\[
(hD_{t}+P(t,x,hD))S(t,s) = 0, \qquad S(t,s)|_{t=s}  = I.
\]
Hence in vector valued notation, for $f_h \in L^2([-\veps,\veps]_s;L^2(\RR^{d-1}_x))$ and $t \in [-\veps,\veps]$,
\begin{equation}\label{duhamel}
\tilde{\psi}_h(t) = \int_{-\veps}^t S(t,s)f_h(s)\,ds.
\end{equation}
Let $\kappa_{t,s}(x,\xi)$ denote the time $t$ value of the integral curve determined by $H_{p_r}$, the Hamiltonian vector field of $p_r(x,\xi) = p(r,x,\xi)$ whose value at time $s$ is $(x,\xi)$.  Taking $\veps>0$ sufficiently small in \eqref{epsdef}, we may assume that $\kappa_{t,s}(x,\xi)$ defines a canonical transformation for $t,s \in [-\veps,\veps]$.  We denote the components of this map in $\RR^{d-1}\times\RR^{d-1}$ as $\kappa_{t,s}(x,\xi) = (x_{t,s}(x,\xi),\xi_{t,s}(x,\xi))$.   Moreover, standard construction (see e.g. \cite[\S 10.2]{zworski}) shows that there exists a phase function $\phi(t,s,x,\eta)$ and a smooth, compactly supported amplitude $a(t,s,x,\eta)$ such that
\begin{equation}\label{parametrix}
\left(S(t,s)f\right)(x) = \frac{1}{(2\pi h)^{d-1}} \iint e^{\frac ih (\phi(t,s,x,\eta)-\langle y ,\eta \rangle)} a(t,s,x,\eta) f(y)\,dy\,d\eta +E f,
\end{equation}
for $\supp(f) \subset [-\veps,\veps]^{d-1}$.  Here $E= E(t,s)$ satisfies $\|E(t,s)\|_{L^2 \to L^2} \lesssim_N h^{N}$ for any $N>0$ (i.e. it is ``$O(h^\infty)$") and given Sobolev embedding, it has a negligible contribution to the estimates in the present work.  Hence we will often make a slight abuse of notation, treating $S(t,s)$ as the same as this oscillatory integral operator.

We further recall that the phase function has the properties
\begin{equation}\label{generatingfunction}
\begin{gathered}
\prtl_t \phi(t,s,x,\eta) + p(t,x,d_x \phi(t,s,x,\eta)) =0,
\\
\kappa_{t,s}(y,\eta) = (x,\xi) \quad \text{if and only if} \quad \xi = d_x \phi(t,s,x,\eta),\; y =  d_\eta \phi(t,s,x,\eta).
\end{gathered}
\end{equation}
The operators $S(t,s)$ are thus semiclassical Fourier integral operators associated to the canonical transformations $\kappa_{t,s}(x,\xi)$.

It is implicit in the work of Koch, Tataru, and Zworski \cite{ktzsemi} that for $s$ fixed, the phase functions $\phi(t,s,x,\xi)$ are of Carleson-Sj\"olin type as defined in \cite[\S2.2]{soggefica}.  Indeed, for $t,s \in [\veps,\veps]$, we have that
\begin{equation}\label{genfcnnondeg}
d_x d_\xi \phi(t,s,x,\xi)   =I+ O( \veps).
\end{equation}
Moreover, it can be checked that
\begin{equation}\label{pgraph}
\left\{ \left(-p(t,x,d_x\phi(t,s,x,\xi)), d_x\phi(t,s,x,\xi)\right): \xi \in \supp(a(t,s,x,\cdot)) \right\}
\end{equation}
is an embedded hypersurface in $T_{(t,x)}^*\RR^{d}$ with nonvanishing principal curvatures, all of the same sign.  This follows as \eqref{genfcnnondeg} implies that this is a local reparameterization of the graph $\{(-p(t,x,\eta), \eta):|\eta| \leq \frac 12\}$ and this is a subset of the strictly convex hypersurface determined by the zero set of the left hand side of \eqref{graph}.

Stein's theorem on Carleson-Sj\"olin phases (see e.g. \cite[Theorem 2.2.1]{soggefica}) thus shows that
\begin{equation}\label{linearbounds}
\left( \int_{-\veps}^{\veps} \left\| S(t,s)g \right\|_{L^q(\RR^{d-1})}^q\,dt \right)^{\frac 1q} \lesssim h^{\frac{d-1}{2}(\frac 12+\frac 1q)-(d-1)}\|\mathscr{F}_h (g)\|_{L^2}\lesssim h^{-\frac{d-1}{2}(\frac 12-\frac 1q)}\|g\|_{L^2}
\end{equation}
Indeed, Stein's theorem treats the case $q = \frac{2(d+1)}{d-1}$, so that the first inequality follows from interpolation with H\"ormander's theorem \cite{HorOsc} for nondegenerate phase functions in $L^2(\RR^{d-1})$ (cf. \eqref{genfcnnondeg}).  The second inequality follows from the semiclassical Plancherel identity.  An application of Minkowski's inequality for integrals then shows that this yields the same linear bounds of the second author in \cite{sogge88}:
\begin{equation*}
\|\tilde{\psi}_h\|_{L^q(\RR^d)} \lesssim \left( \int_{-\veps}^{\veps} \left\| \int_{-\veps}^t S(t,s)f_h(s)\,ds \right\|_{L^q(\RR^{d-1})}^q\,dt \right)^{\frac 1q}\lesssim h^{-\frac{d-1}{2}(\frac 12-\frac 1q)}\|\psi_h\|_{L^2(M)}.
\end{equation*}

In our case, linear bounds of this type will play a role, but ultimately the key is bilinear estimates, specifically those due to H\"ormander \cite{HorOsc} when $d=2$ and S. Lee \cite{leebilinear} and the epsilon removal lemma of the authors \cite{blairsogge} when $d \geq 3$.  For now, we remark that the reductions above mean that it now suffices to prove estimates on
\begin{equation}\label{ehtildesquared}
\left\|(\tilde{\psi}_h(t))^2\right\|_{L^{\frac q2}([-\veps,\veps]_t\times\RR^{d-1}_x)} = \left\|\left(\int_{-\veps}^t S(t,s)f_h(s)\,ds\right)^2\right\|_{L^{\frac q2}([-\veps,\veps]_t\times\RR^{d-1}_x)},
\end{equation}
showing they are bounded by the square of the right hand sides in \eqref{higherdimbound}, \eqref{twodimboundL4}.

\subsection{An Egorov theorem for critical and near-critical symbols}
We conclude this section with an Egorov type theorem in the symbol class $S_{1/2-\delta}(1)$, $\delta \in [0,\frac 12]$ (see \cite[p.73]{zworski} for the notation) which are symbols of the form \eqref{thetasymbol} with $\theta = h^{1/2-\delta}$ below.  The $\delta =0$ case is considered critical as the usual symbolic calculus does not furnish terms of higher order in $h$.  For example, the usual stationary phase expansions one often uses to prove the calculus do not yield these higher order terms.  Consequently, one cannot apply the usual Egorov theorem to such symbols.

Here we show a version of the Egorov theorem for $h$-pseudodifferential operators (PDO) with symbols in these classes.  As observed in \cite[Ch.VIII,\S8]{taylorpdo} and \cite{christiansonmonodromy}, there are Egorov theorems for classical PDO in these critical symbol classes, but here we are interested in a more qualitative version of the theorem.  While we cannot show that conjugation by an $h$-Fourier integral operator (FIO) results in a strict propagation of support for the symbols via the associated canonical transformation, the symbol is rapidly decreasing on a scale of at least $h^{-1/2}$ in terms of the distance to this propagated region, which is enough for our applications.  The theorem not only applies to the operator $S(t,s)$ defined above, but to more general local $h$-FIO.
\begin{theorem}\label{thm:egorov}
Let $S$ be a local $h$-Fourier integral operator
\[
(Sf)(x) = \frac{1}{(2\pi h)^n}\iint_{\RR^{2n}} e^{\frac ih (\phi(x,\eta)-\langle y,\eta \rangle)} a(x,\eta) f(y)\,dyd\eta,
\]
for some compactly supported symbol $a \in S(1)$ over $\RR^{2n}$ and associated to a canonical transformation $\kappa:\RR^{2n} \to \RR^{2n}$ in that $\kappa(y,\eta) = (x,\xi)$ if and only if $\xi=d_x \phi(x,\eta)$ and $y=d_\eta \phi(x,\eta)$.  Furthermore, let $B$ be an $h$-pseudodifferential operator with symbol satisfying
\begin{equation}\label{thetasymbol}
\left| \prtl^\alpha b(y,\xi) \right| \lesssim_\alpha \theta^{-|\alpha|} \qquad \text{where } \theta = h^{\frac 12-\delta},\, \delta \in \mbox{$[0,\frac 12]$}.
\end{equation}
Then $C=SBS^*$ defines an $h$-pseudodifferential operator with symbol $c(x,\rho)$ also satisfying \eqref{thetasymbol}.  If in addition, $\supp(b) \subset\mathscr{D}$ for some compact set $\mathscr{D} \subset \RR^n$, then for every $N \geq 0$,
\begin{equation}\label{symbolbdwithkappa}
\left| \prtl^\alpha c(x,\rho) \right| \lesssim_{\alpha,N} \theta^{-|\alpha|} (1+h^{-1}\theta d(x,\rho;\kappa(\mathscr{D})))^{-N},
\end{equation}
with $d(x,\rho;\kappa(\mathscr{D}))$ denoting the distance to $\kappa(\mathscr{D})$, the image of $\mathscr{D}$ under $\kappa$.  Furthermore, for any $M>0$ there exists a symbol $c_M(y,\eta)$ satisfying \eqref{thetasymbol} such that $\supp(c_M) \subset \kappa(\mathscr{D})$ and for every $N \geq 0$,
\begin{equation}\label{compactapprox}
\left| \prtl^\alpha (c-c_M)(y,\eta) \right| \lesssim_{\alpha,N} \theta^{-|\alpha|} (h\theta^{-2})^M(1+h^{-1}\theta d(y,\eta;\kappa(\mathscr{D})))^{-N}.
\end{equation}
\end{theorem}

Tracing through the proof, it is seen that the implicit constants appearing in \eqref{symbolbdwithkappa} and \eqref{compactapprox} do not depend on $\delta$, but only on those in \eqref{thetasymbol} and the derivative bounds for $a$, $\kappa$, and $\phi$.  In this work, we are mainly interested applying this theorem in the $\delta=0$ case, but record the more general case as a point of interest.  We also stress that $B$ is defined by taking the standard quantization of the symbol $b(y,\xi)$.

\begin{proof}
The composition $C=SBS^*$ has a Schwartz kernel $K(x,\tilde{x})$ given by an oscillatory integral.  For sufficiently regular $f$, the compact support of the symbols ensures that we may write
\[
(Cf)(x) = \frac{1}{(2\pi h)^n} \int e^{\frac ih \langle x,\rho \rangle}\left( \int e^{\frac ih \langle \tilde{x}-x,\rho \rangle} K(x,\tilde{x})\,d\tilde{x} \right)\mathscr{F}_h(f)(\rho)\,d\rho.
\]
The expression in parentheses determines the symbol $c(x,\rho)$ and can be written as
\begin{multline*}
c(x,\rho)=\frac{1}{(2\pi h)^{3n}} \int_{\RR^{6n}}e^{\frac ih (\langle y-\tilde{y},\xi \rangle - \langle y,\eta \rangle + \langle \tilde{y}, \tilde{\eta} \rangle + \phi(x,\eta) - \phi(\tilde{x},\tilde{\eta}) + \langle x-\tilde{x},\rho \rangle)}\\
\times a(x,\eta)\overline{a(\tilde{x},\tilde{\eta})}b(y,\xi)
\,d(y,\tilde{y},\eta ,\tilde{\eta},\xi ,\tilde{x})
\end{multline*}
where the integrations here must be interpreted in the sense of distributions since the amplitude is independent of $\tilde{y}$.  A standard limiting procedure shows that this presents no difficulty, allowing us to integrate in the $\tilde{y}$ variable first, which yields the distribution $(2\pi h)^{n}\delta(\tilde{\eta}-\xi)$.  Hence $c(x,\rho)$ can be written as
\begin{equation*}
\frac{1}{(2\pi h)^{2n}} \int_{\RR^{4n}}
e^{\frac ih(\phi(x,\eta)-\phi(\tilde{x},\tilde{\eta}) + \langle \tilde{x}-x,\rho\rangle + \langle y, \tilde{\eta}-\eta\rangle )}
a(x,\eta)\overline{a(\tilde{x},\tilde{\eta})}b(y,\tilde{\eta})\,d\eta d \tilde{\eta}dy d\tilde{x},
\end{equation*}
and the integration here is no longer formal. We rewrite this as
\begin{multline*}
\frac{1}{(2\pi h)^{2n}} \int_{\RR^{4n}}
e^{\frac ih(\phi(x,\eta)-\phi(\tilde{x}+x,\tilde{\eta}+\eta) + \langle \tilde{x},\rho\rangle + \langle y, \tilde{\eta} \rangle )}\\\times
a(x,\eta)\overline{a(\tilde{x}+x,\tilde{\eta}+\eta)}b(y,\tilde{\eta}+\eta)\,d\eta d \tilde{\eta}dy d\tilde{x}.
\end{multline*}

For each $(x,\tilde{x},\tilde{\eta},\rho)$, we make a further change of variables $(y,\eta)\mapsto (w,\sigma)$ with
\begin{align}
w&= y-\int_0^1 d_\eta \phi(x+s\tilde{x},\eta+s \tilde{\eta})\,ds = y-d_\eta \phi(x+\tilde{x},\eta+\tilde{\eta}) + O(|(\tilde{x},\tilde{\eta})|)\label{wdef}\\
\sigma &= \rho - \int_0^1 d_x \phi(x+s\tilde{x},\eta + s \tilde{\eta})\,ds = \rho - d_x \phi(x,\eta) + O(|(\tilde{x},\tilde{\eta})|)\label{taudef}
\end{align}
so that the phase reads $\langle (\tilde{x},w),(\sigma,\tilde{\eta})\rangle$.  The invertibility of the mixed Hessian $d_x d_\eta \phi$ ensures that locally, $(y,\eta)\mapsto (w,\sigma)$ defines an invertible transformation.  The last expression in \eqref{wdef} means that we may write $y=d_\eta\phi(x+E_1, \eta + \tilde{\eta})$  where $E_1(x,\tilde{x},w, \sigma,\tilde{\eta},\rho)=O(|(w, \tilde{x}, \tilde{\eta})|)$ as $(w, \tilde{x}, \tilde{\eta}) \to 0$.  Hence
\[
 b(y,\eta + \tilde{\eta}) = b(d_\eta\phi(x+E_1, \eta + \tilde{\eta}), \eta + \tilde{\eta})
 = (b\circ \kappa^{-1})(x+E_1,d_x \phi(x+E_1, \eta + \tilde{\eta}))
\]
where $y,\eta$ are understood as functions of $(x,\tilde{x},w, \sigma,\tilde{\eta},\rho)$.  Moreover, the last expression in \eqref{taudef} means that $d_x \phi(x+E_1, \eta + \tilde{\eta}) = \rho + E_2$, with $E_2(x,\tilde{x},w, \sigma,\tilde{\eta},\rho) = O(|(w, \tilde{x}, \tilde{\eta},\sigma)|)$ near 0.  Up to $O(h^\infty)$ error, $c(x,\rho)$ can thus be rewritten as
\begin{equation}\label{stdform}
\frac{1}{(2\pi h)^{2n}} \int_{\RR^{4n}}
e^{\frac ih\langle (\tilde{x},w),(\sigma,\tilde{\eta}) \rangle}
A(x,\tilde{x},w, \sigma,\tilde{\eta},\rho)(b\circ \kappa^{-1})(x+E_1,\rho + E_2)\,d\eta d \tilde{\eta}dy d\tilde{x}
\end{equation}
where $A$ and its partial derivatives are uniformly bounded in $h$.

As observed in \cite[p.51]{zworski}, the phase function in \eqref{stdform} can be rewritten as
\[
\langle (\tilde{x},w),(\sigma,\tilde{\eta})\rangle =\frac 12\langle Q (\tilde{x},w,\sigma,\tilde{\eta}),(\tilde{x},w,\sigma,\tilde{\eta})\rangle,
\text{ where } Q= \begin{bmatrix} 0 & I\\ I & 0 \end{bmatrix} = Q^{-1}.
\]
At this point, it can be treated using the method in \cite[Theorem 3.13]{zworski}.  We momentarily treat $x,\rho$ as fixed, and let $u(\tilde{x},w,\sigma,\tilde{\eta}) $ be the amplitude in \eqref{stdform} formed by fixing these variables.  Let $\zeta\in \RR^{4n}$ denote variables which are dual to $(\tilde{x},w,\sigma,\tilde{\eta})$, then using the classical Fourier transform, \eqref{stdform} can be rewritten as
\[
J(h,u)=\frac{1}{(2\pi)^{4n}} \int_{\RR^{4n}} e^{-\frac {ih}2 \langle Q\zeta, \zeta \rangle} \widehat{u}(\zeta) \,d\zeta
.
\]
Applying Taylor's formula in $h$ we have for every $M \in \mathbb{N}$, the differential operator $\widetilde{P}=-i\langle D_{(\tilde{x},w)}, D_{(\sigma,\tilde{\eta})}\rangle$ can be used to write
\begin{equation}\label{Jhuexpansion}
 J(h,u) = \sum_{k=0}^{M-1}\frac{h^k}{k!} J(0,\widetilde{P}^k u) + \frac{h^M}{(M-1)!}\int_0^t (1-t)^{M-1}J(th,\widetilde{P}^Mu)\,dt.
\end{equation}
The sum in $k$ can be rewritten as $\sum_{k=0}^{M-1}(h\theta^{-2})^k\tilde{c}_k (x,\rho)$ with $|\prtl^\alpha \tilde{c}_k | \lesssim \theta^{-|\alpha|},$ and $\supp(c_k) \subset \supp(b\circ \kappa^{-1}) = \kappa(\supp(b)).$

A change of variables $\zeta \mapsto t^{-\frac 12}\zeta$ in the remainder $J(th,\widetilde{P}^Mu)$ yields
\begin{equation}\label{statphaseremainder}
 J(th,\widetilde{P}^Mu) = \frac{1}{(2\pi \sqrt{t})^{4n}} \int_{\RR^{4n}} e^{-\frac {ih}2 \langle Q\zeta, \zeta \rangle} \widehat{\widetilde{P}^M u}(t^{-\frac 12}\zeta) \,d\zeta\\
\end{equation}
We are thus led to study integrals of the following form, for $t \in (0,1]$
\begin{equation}\label{modelamplitude}
\tilde{c}_t(x,\rho)=\frac{1}{(2\pi h)^{2n}} \int_{\RR^{4n}} e^{\frac ih \langle (\tilde{x},w),(\sigma,\tilde{\eta})\rangle }
\widetilde{A}(x,t^{\frac 12}\tilde{x},t^{\frac 12}w,t^{\frac 12}\sigma,t^{\frac 12}\tilde{\eta},\rho) d\tilde{x} dw d\sigma d\tilde{\eta}.
\end{equation}
for some amplitude $\widetilde{A}(x,\tilde{x},w,\sigma,\tilde{\eta},\rho)$ satisfying $|\prtl^\alpha \widetilde{A}|\lesssim_\alpha \theta^{-|\alpha|}$ and
\begin{multline*}
\{(x,\tilde{x},w,\sigma,\tilde{\eta},\rho) :\widetilde{A}(x,\tilde{x},w,\sigma,\tilde{\eta},\rho) \neq 0 \}\\ \subset \{(x,\tilde{x},w,\sigma,\tilde{\eta},\rho) :(b \circ \kappa^{-1})(x+E_1,\rho+E_2) \neq 0\} .
\end{multline*}
Indeed, the identity in \eqref{statphaseremainder}, combined with the Fourier transform, shows that the remainder term in \eqref{Jhuexpansion} can be expressed as
\[
\frac{(h\theta^{-2})^M}{(M-1)!} \int_0^1 \tilde{c}_t(x,\rho)(1-t)^{M-1}\,dt
\]
for some integral $\tilde{c}_t(x,\rho)$.  Moreover, \eqref{stdform} is of the form \eqref{modelamplitude} with $t=1$, so studying these integrals provides a unified approach to \eqref{symbolbdwithkappa} and \eqref{compactapprox}.

We now show that $(\theta D_{x,\rho})^\alpha \tilde{c}_t(x,\rho)$ has the effect of replacing the amplitude in the definition by one of the same regularity.  By induction, it suffices to prove this for $|\alpha|=1$.  This is clear for any such weighted derivative applied to the amplitude, hence it suffices to consider the effect of such derivatives on the phase. Recall that $\tilde{x}, \tilde{\eta}$ are independent of $(x,\rho)$ even though $(w,\sigma)$ are not.  Hence
\[
\theta D_{x_i}e^{\frac ih \langle (\tilde{x},w),(\sigma,\tilde{\eta})\rangle} = \left( \prtl_{x_i}\sigma \cdot(\theta D_\sigma) +
\prtl_{x_i}w \cdot(\theta D_w)\right)e^{\frac ih \langle (\tilde{x},w),(\sigma,\tilde{\eta})\rangle},
\]
and integration by parts completes the proof of the claim.  Derivatives in $\rho$ are handled analogously.

To see \eqref{symbolbdwithkappa}, \eqref{compactapprox} with $N=0$, it remains to show that $|\tilde{c}_t(x,\rho)|$ is uniformly bounded in $h$.  Note that the first order differential operator
\begin{equation}\label{firstorderdo}
\frac{1+(\sigma,\tilde{\eta},\tilde{x},w)\cdot D_{(\tilde{x},w,\sigma,\tilde{\eta})}}{1+h^{-1}|(\sigma,\tilde{\eta},\tilde{x},w)|^2}
=\frac{1+(h^{-\frac 12}(\sigma,\tilde{\eta},\tilde{x},w))\cdot(h^{\frac 12}D_{(\tilde{x},w,\sigma,\tilde{\eta})})}{1+h^{-1}|(\sigma,\tilde{\eta},\tilde{x},w)|^2}
\end{equation}
preserves the exponential factor in \eqref{modelamplitude}.  Writing the differential operator in the second fashion illustrates that the transpose of this operator applied to the class of amplitudes here yields an amplitude of the same regularity (in that there is no loss in $h$).  Integration by parts sufficiently many times yields an amplitude which is $O((1+h^{-1}|(\sigma,\tilde{\eta},\tilde{x},w)|^2)^{-4n})$, so that integration in these variables counterbalances the loss of $h^{-2n}$ in front of \eqref{modelamplitude}.  Alternatively, one could verify \eqref{symbolbdwithkappa}, \eqref{compactapprox}  when $N=0$ by changing variables $(\tilde{x},w,\sigma,\tilde{\eta}) \mapsto h^{1/2}(\tilde{x},w,\sigma,\tilde{\eta})$ from the outset in \eqref{modelamplitude}.

To show \eqref{symbolbdwithkappa}, \eqref{compactapprox} for $N\geq 1$, we observe that if $d(x,\rho;\kappa(\mathscr{D})) \geq h\theta^{-1}$, then
\begin{equation}\label{ctdecay}
 |c_t(x,\rho)| \lesssim_N \left(h^{-1}\theta d(x,\rho;\kappa(\mathscr{D}))\right)^{-N}.
\end{equation}
Indeed, since $(b \circ \kappa^{-1})(x+E_1,\rho+E_2) \neq 0$ over the domain of integration in \eqref{modelamplitude}, our previous observation that $E_1,E_2 = O(|(\tilde{x},w,\sigma,\tilde{\eta})|)$ means that for such $(x,\rho)$,
\begin{equation*}
 h\theta^{-1}\leq  d(x,\rho;\kappa(\mathscr{D})) \leq |(E_1,E_2)| \lesssim |t^{\frac 12} (\tilde{x},w,\sigma,\tilde{\eta})| \lesssim |(\tilde{x},w,\sigma,\tilde{\eta})|,
\end{equation*}
Hence the phase function in \eqref{modelamplitude} has no critical points, and each integration by parts with respect to $D_{(\tilde{x},w,\sigma,\tilde{\eta})}$ yields a gain of at least $O(h\theta^{-1}|(\tilde{x},w,\sigma,\tilde{\eta})|^{-1})$.  Integrating by parts sufficiently many times with this operator followed by a successive integration by parts using the operator in \eqref{firstorderdo} thus implies \eqref{ctdecay}.
\end{proof}

\begin{remark}
It is interesting to note that if one does not assume $\supp(b)\subset \mathscr{D}$, but instead the weaker decay condition
\[
\left| \prtl^\alpha b(x,\xi) \right| \lesssim_{\alpha,N} \theta^{-|\alpha|} (1+\theta^{-1} d(x,\xi;\mathscr{D}))^{-N},
\]
then the symbol of $c$ satisfies the same bound, but with $\mathscr{D}$ replaced by $\kappa(\mathscr{D})$.  This is a matter of replacing the differential operator in \eqref{firstorderdo} by
\[
\frac{1+(\theta^{-1}(\sigma,\tilde{\eta},\tilde{x},w))\cdot(h\theta^{-1} D_{(\tilde{x},w,\sigma,\tilde{\eta})})}{1+\theta^{-2}|(\tilde{x},w,\sigma,\tilde{\eta})|^2}
\]
and proceeding in a similar fashion to the $t=1$ case in the analysis of \eqref{modelamplitude} above.
\end{remark}

\section{Microlocal Kakeya-Nikodym Norms and the Outline of the Proof}\label{sec:MKNoutline}
Let $J$ be the negative integer satisfying $2^{J} \geq \sqrt{h} >2^{J-1}$. Given $\nu \in 2^J\mathbb{Z}^{2(d-1)}$, let $b_\nu^J$ be a smooth bump function such that for some $\tilde \veps>0$ sufficiently small
\begin{equation}\label{bjsupp}
\begin{gathered}
\supp(b_\nu^J)\subset \nu + \left[-\left(\frac 12+\tilde\veps\right)2^J,\left(\frac 12+\tilde\veps\right) 2^J\right]^{2(d-1)},
\\
\sum_{\nu \in 2^J\mathbb{Z}^{2(d-1)} } b_\nu^J(x,\xi) = 1 \quad \text{and} \quad \left|\prtl^\alpha_{y,\eta} b_\nu^J (y,\eta)\right| \lesssim_{\alpha} 2^{-J|\alpha|} \approx h^{-\frac{|\alpha|}{2}}.
\end{gathered}
\end{equation}
We then define $B_\nu^J(0)$ (to be extended to a family of operators $\{B_\nu^J(t) \}_{t \in [-\veps,\veps]}$), to be the semiclassical $h$-pseudodifferential operator with symbol $b_\nu^J \in S_{1/2}(1)$
\begin{equation}\label{bdef}
(B_\nu^J(0)g)(x)=\frac{1}{(2\pi h)^{d-1}}\iint_{\RR^{2(d-1)}} e^{\frac ih \langle x-y,\eta\rangle}b_\nu^J\left(x,\xi\right)g(y)\,dyd\eta.
\end{equation}
We remark that by a standard rescaling argument and the Calder\'on-Vaillancourt theorem, $B_\nu^J(0)$ is bounded on $L^2(\RR^{d-1}_x)$ (as is any operator with symbol in $S_{1/2}(1)$).  Extend $B_\nu^J(0)$ to a family of operators by defining
\[
B_\nu^J(t)= S(t,0) B_\nu^J(0)S(0,t).
\]

Given $g\in L^2([-\veps,\veps]_s;L^2(\RR^{d-1}_x))$, let $T$ denote the operator defined in vector valued notation
\[
(Tg) (t)= \int_{-\veps}^t S(t,s)g(s)\,ds, \text{ that is, } (Tg)(t,x) = \int_{-\veps}^t \big(S(t,s)g(s,\cdot)\big)(x)\,ds,
\]
so that $g(s),(Tg) (t) \in L^2(\RR^{d-1}_x)$ for $t,s\in[-\veps,\veps]$.  As suggested in \eqref{ehtildesquared}, the main strategy for proving the estimates in \eqref{higherdimbound}, \eqref{twodimboundL4} will be to prove $L^{q/2}$ bounds on $(\tilde{\psi}_h)^2$, leading us to consider $(Tf_h)^2 $.  Since $f_h(s) = \sum_\nu B_\nu^{J}(s)f_h(s)$, we may write
\begin{align}
(Tf_h(t))^2 &= \sum_{\nu,\nu'} \left( \int_{-\veps}^t S(t,s)B_\nu^J(s)f_h(s)\,ds\right)\left( \int_{-\veps}^t S(t,r)B_{\nu'}^J(r)f_h(r)\,dr\right)\notag\\
&= \sum_{\nu,\nu'} T(B_{\nu}^Jf_h)(t) T(B_{\nu'}^Jf_h)(t)\label{nusum}
\end{align}
where $B_{\nu}^Jf_h$ abbreviates the vector valued function $s\mapsto B_{\nu}^J(s)f_h(s)$.

We next consider the family of dyadic cubes in $\RR^{2(d-1)}$, $\tau^j_\mu$ formed by taking the cube $[-2^{j-1},2^{j-1})^{2(d-1)}$, and translating by an element $\mu \in 2^j \mathbb{Z}^{2(d-1)}$.  As in \cite{tvv}, two dyadic cubes are said to be \emph{close}, writing $\tau^j_\mu \sim \tau^j_{\mu'}$, if they are not adjacent, but have adjacent parents of sidelength $2^{j+1}$.  We use this to organize the sum in \eqref{nusum}: given a pair $(\nu,\nu')$ such that $2^{-J}|\nu-\nu'|$ is sufficiently large, there exist close dyadic cubes $\tau^j_\mu \sim \tau^j_{\mu'}$ such that $\nu \in \tau^j_\mu$ and $\nu' \in\tau^j_{\mu'}$ (which means $|\nu-\nu'|\approx 2^j$), allowing us to write
\begin{multline}\label{whitneyorg}
\sum_{\nu,\nu'} T(B_{\nu}^Jf_h) T(B_{\nu'}^Jf_h) =  \sum_{(\nu,\nu')\in \Xi_J} T(B_{\nu}^Jf_h) T(B_{\nu'}^Jf_h)+\\
\sum_{j=J+1}^{0} \sum_{(\nu,\nu') \in \tau^j_\mu \times \tau^j_{\mu'}: \tau^j_\mu\sim \tau^j_{\mu'}} T(B_{\nu}^Jf_h) T(B_{\nu'}^Jf_h) + O(h^\infty),
\end{multline}
where we let $\Xi_J$ denote pairs $(\nu,\nu')$ lying in adjacent cubes.  Note that by \eqref{frequencyerror}, we may limit ourselves to cubes of distance less than 1 at the cost of the negligible $O(h^\infty)$ error term on the right. When $0\leq j \leq J+1$, we define
\[
\Xi_j:=\{(\mu,\mu'):\tau^j_\mu \sim \tau^j_{\mu'}\},
\]
noting the slight variation in these definitions.  For $\mu \in 2^j\mathbb{Z}^{2(d-1)}$, we define $B_\mu^j(0)$ as the $h$-PDO with symbol
\[
b_\mu^j(y,\eta) := \sum_{\nu \in \tau^j_\mu} b_{\nu}^J(y,\eta).
 \]
For convenience, set
\begin{equation}
  \mathscr{D}_\mu^j := \supp(b_\mu^j).
\end{equation}
Taking $\tilde \veps$ small in \eqref{bjsupp}, we have for $(\mu,\mu')\in \Xi_j$, we have
\begin{equation}\label{bjseparation}
d\big(\mathscr{D}_\mu^j ,\mathscr{D}_{\mu'}^j \big)\approx 2^j, \qquad 0 \geq j \geq J+1.
\end{equation}
Also define
\[
B_\mu^j(s) := \sum_{\nu \in \tau^j_\mu} B_{\nu}^J(s)=S(s,0)B_\mu^j(0)S(0,s),
\]
and denote the symbol of $B^j_\mu(s)$ as $b_\mu^j(s,y,\eta)$.  Given Theorem \ref{thm:egorov},
\begin{equation}\label{actualdecayprep}
\left| \prtl^\alpha_{y,\eta} b_\mu^j(s,y,\eta)\right| \lesssim_{\alpha,N} h^{-\frac{|\alpha|}{2}} \left(1+h^{-\frac 12}d\left(y,\eta;\kappa_{s,0}\left( \mathscr{D}_\mu^j\right)\right)\right)^{-N},
\end{equation}
and the same holds when $j=J$.

With this notation, the first sum on the right in \eqref{whitneyorg} can be rewritten as
\begin{equation}\label{whitneyorg2}
\sum_{j=J+1}^{0} \sum_{(\mu, \mu') \in \Xi_j} T(B_{\mu}^jf_h) T(B_{\mu'}^jf_h).
\end{equation}
Hence $\|(\tilde{\psi}_h)^2\|_{L^{\frac q2}([-\veps,\veps]_t \times \RR^{d-1}_x)} $ is bounded above by
\begin{equation}\label{whitneyorgnorm}
\sum_{j=J+1}^0 \left\|\sum_{(\mu, \mu') \in \Xi_j } T(B_{\mu}^jf_h) T(B_{\mu'}^jf_h)\right\|_{L^{\frac q2}} + \left\|\sum_{(\nu,\nu')\in \Xi_J} T(B_{\nu}^Jf_h) T(B_{\nu'}^Jf_h)\right\|_{L^{\frac q2}}.
\end{equation}

To estimate \eqref{whitneyorgnorm}, we follow the strategy in \cite{blairsoggerefined}, first proving almost orthogonality of the $T(B_{\mu}^jf_h) T(B_{\mu'}^jf_h)$ for $(\mu, \mu') \in \Xi_j$, then using bilinear estimates to bound each product.  However, to make this rigorous, we will need to instead consider a family of $h$-PDO for $|s| \leq \veps$ where $\veps$ is as in \eqref{epsdef}
\begin{equation}\label{Bfamily}
\left\{ B_{\mu,\omega}^j(s): \omega \in \Omega, s\in [-\veps,\veps]\right\}
\end{equation}
where $\Omega$ is a finite index set and each corresponding symbol $b_{\mu,\omega}^j(s,y,\eta)$ in the family satisfies the same bound as in \eqref{actualdecayprep}
\begin{equation}\label{actualdecay}
\left| \prtl^\alpha_{y,\eta} b_{\mu,\omega}^j(s,y,\eta)\right| \lesssim_{\alpha,N} h^{-\frac{|\alpha|}{2}} \left(1+h^{-\frac 12}d\left(y,\eta;\kappa_{s,0}\left(\mathscr{D}_{\mu}^j\right)\right)\right)^{-N},
\end{equation}
and for each $\alpha$ and $N$, the implicit constants here are uniformly bounded in $\mu,\omega,j$.  This collection will be defined below.  A $B_{\mu,\omega}^j(s)$ can thus be viewed as akin to a $B_{\mu}^j(s)$, but whose symbol as been distorted (though compact support of the symbol may be lost when $s=0$).

Before getting to the crux of our argument, we prove two lemmas.
\begin{lemma}\label{lem:aosums}
Suppose $\omega_\mu\in \Omega$ is any sequence parameterized by $\mu \in 2^j\mathbb{Z}^{2(d-1)}$. Then for any $J \leq j \leq 0$,
\begin{equation}\label{aoomega}
\sum_\mu \left\| B_{\mu,\omega_\mu}^j f_h\right\|_{L^2([-\veps,\veps]_t \times \RR^{d-1}_x))}^2 \lesssim \|f_h\|_{L^2}^2.
\end{equation}
\end{lemma}
\begin{proof}
 Observe that it suffices to show that for each $s\in [-\veps,\veps]$
 \begin{equation*}
  \sum_\mu \left\| B_{\mu,\omega_\mu}^j(s) f_h(s)\right\|_{L^2(\RR^{d-1})}^2  \lesssim \|f_h(s)\|_{L^2(\RR^{d-1})}^2,
 \end{equation*}
 and integrate both sides over $s \in [-\veps,\veps]$. By Khintchine's inequality, this reduces to showing that for an arbitrary sequence $\{\upsilon_\mu\}_\mu$ taking values $\upsilon_\mu = \pm 1$
 \begin{equation*}
  \left\|   \sum_\mu \upsilon_\mu B_{\mu,\omega_\mu}^j(s) f_h(s)\right\|_{L^2(\RR^{d-1})} \lesssim \|f_h(s)\|_{L^2(\RR^{d-1})},
 \end{equation*}
where the implicit constant on the right is independent of the sequence.  By the aforementioned application of the Calder\'on-Vaillancourt theorem, this further reduces to showing that the symbol
\begin{equation}\label{symbolsum}
 \sum_\mu \upsilon_\mu b_{\mu,\omega_\mu}^j(s,y,\eta)
\end{equation}
is in $S_{1/2}(1)$ with bounds uniform in the sequence $\{\upsilon_\mu\}_\mu$.  This in turn follows from verifying that for $N \geq 2d$,
\begin{equation}\label{supportsum}
 \sum_\mu \left(1+h^{-\frac 12}d\left(y,\eta;\kappa_{s,0}\left(\mathscr{D}_{\mu}^j\right)\right)\right)^{-N} \lesssim \sum_\mu \left(1+2^{-j}d\left(\kappa_{0,s}(y,\eta);\mathscr{D}_{\mu}^j\right)\right)^{-N} \lesssim 1.
\end{equation}
Indeed, if this holds, then any weighted derivative $(h^{1/2}\prtl_{y,\eta})^\alpha$ of the symbol in \eqref{symbolsum} is $O(1)$ by \eqref{actualdecay} and the triangle inequality.

To see \eqref{supportsum}, note that the first inequality follows since $\kappa_{s,0}$ and its inverse are Lipschitz maps and $2^{-j} \leq h^{-1/2}$.  For the second inequality, observe that
\begin{equation*}
  d\left(\kappa_{0,s}(y,\eta);\mathscr{D}_{\mu}^j\right) \leq 4 |\kappa_{0,s}(y,\eta)-\mu| \;\text{ if } \; d\left(\kappa_{0,s}(y,\eta);\mathscr{D}_{\mu}^j\right) \geq 2^{j+3}.
\end{equation*}
The rest of \eqref{supportsum} then follows as a consequence of
\begin{equation*}
  \sum_{\mu \in 2^{j}\mathbb{Z}^{2(d-1)}} \left(1+2^{-j}|\kappa_{0,s}(y,\eta)-\mu| \right)^{-N} \lesssim 1.
\end{equation*}
\end{proof}

We now define the microlocal Kakeya-Nikodym norm of $f_h$ as
\[
\|f_h\|_{MKN}^2 := \sup_{J \leq j \leq 0} \sup_{\omega \in \Omega} \sup_{\mu \in 2^j\mathbb{Z}^{2(d-1)}} 2^{-j(d-1)}\int_{-\veps}^{\veps} \|B_{\mu,\omega}^j(s)f_h(s)\|_{L^2(\RR^{d-1})}^2\,ds.
\]
\begin{lemma}\label{lem:MKNandKN}
Using the norms in \eqref{KNnorm} involving averages over tubes, we have that \begin{equation}\label{MKNandKN}
\|f_h\|_{MKN} \lesssim \|\psi_h\|_{KN}.
\end{equation}
\end{lemma}

\begin{proof} The bound \eqref{MKNandKN} will follow by showing that for any $j,\mu,\omega$
\[
2^{-j(d-1)}\int_{-\veps}^{\veps} \|B_{\mu,\omega}^j(s)f_h(s)\|_{L^2(\RR^{d-1})}^2\,ds \lesssim \sup_{\gamma \in \varPi} h^{-\frac{d-1}{2}} \int_{\mathcal{T}_{h^{1/2}}(\gamma)} |\psi_h(t,x)|^2dxdt.
\]
In what follows, let $\pi:T^*\RR^{d-1} \to \RR^{d-1}$ denote the projection $\pi(x,\xi) = x$.  Observe that the integral curves of the Hamiltonian vector field $H_{p_t}$ are nonlinear reparameterizations of unit speed geodesics in $T^*\RR^d$, which are integral curves of the Hamiltonian vector field determined by $\frac 12(\sum_{i,j}g^{ij}(z)\zeta_i\zeta_j -1)$.  So while
\begin{equation*}
\mathscr{E}_\mu^j :=\cup_{s\in[-\veps,\veps]}\pi(\kappa_{s,0}(\mathscr{D}_{\mu}^j))
\end{equation*}
does not define a tubular neighborhood, it is contained in a tubular neighborhood of width $O(2^j)$ about the unit speed geodesic segment which intersects the $s=0$ plane with coordinates $\mu \in T^*(\{s=0\})$.

We now claim that there exists a cover of our coordinate chart by family of (not necessarily tubular) neighborhoods about geodesic segments $\{ \widetilde{\mathcal{T}}_\mathbf{k}\}_{\mathbf{k}}$ indexed by $\mathbf{k} \in 2^J\mathbb{Z}^{d-1}$ ($2^J \approx h^{1/2}$ as above) and an associated partition of unity $\sum \chi_{\mathbf{k}}^2 = 1$ in the chart with $\supp(\chi_{\mathbf{k}}) \subset \widetilde{\mathcal{T}}_\mathbf{k}$ with the following properties for some $C$ uniform and $\mathbf{k}_0$ fixed:
\begin{gather}
 d(\widetilde{\mathcal{T}}_\mathbf{k},\mathscr{E}_\mu^j) \approx |\mathbf{k}-\mathbf{k}_0| \text{ whenever }
 d\big(\widetilde{\mathcal{T}}_\mathbf{k}, \mathscr{E}_\mu^j \big) \geq C2^j,
\label{distant}\\
\#\left\{\mathbf{k}: d\left(\widetilde{\mathcal{T}}_\mathbf{k}, \mathscr{E}_\mu^j \right) < C2^j\right\} = O\left((h^{-1/2}2^j)^{d-1}\right),\label{cardinality}\\
\text{each $\widetilde{\mathcal{T}}_\mathbf{k}$ can be covered by $O(1)$ tubular neighborhoods $\tubeh$},\label{finitecover}\\
\widetilde{\mathcal{T}}_\mathbf{k}\cap \mathcal{T}_{\tilde{\mathbf{k}}} = \emptyset \text{ and } d(\widetilde{\mathcal{T}}_\mathbf{k},\widetilde{\mathcal{T}}_{\tilde{\mathbf{k}} })\approx |\mathbf{k}-\tilde{\mathbf{k}}|  \quad\text{ if } \quad|\mathbf{k}- \tilde{\mathbf{k}}| \geq 2^{J+3}. \label{mutualsep}
\end{gather}
To see that this cover exists, we momentarily take Fermi coordinates $(z_1,z')\in \RR^d$ adapted to a hypersurface orthogonal to the geodesic segment $\{\kappa_{s,0}(\mu): |s|\leq \veps\}$, so that $\{z'=0\}$ identifies the hypersurface and $r\mapsto (r,z')$ always parameterizes a unit speed geodesic.  In particular, we assume that $r \mapsto (r,0)$ parameterizes the geodesic segment given by $s\mapsto \kappa_{s,0}(\mu)$ so that $\mathscr{E}_\mu^j \subset \{z: |z'| \lesssim 2^j\}$.  In these coordinates, we take a partition of unity $\sum \chi_{\mathbf{k}}^2 = 1$ so that with $\mathbf{k} = (k_2,\dots,k_d)$
\begin{equation*}
\supp(\chi_{\mathbf{k}} ) \subset \widetilde{\mathcal{T}}_\mathbf{k} :=\left\{(z_1,z'): |z_1| \lesssim \veps, \; \max_{2\leq i\leq d}|z_i - k_i| \leq 2^{J+1}  \right\}.
\end{equation*}
The property \eqref{mutualsep} is immediate in these coordinates.  The set $\widetilde{\mathcal{T}}_\mathbf{k}$ is a neighborhood about the geodesic segment $\gamma_{\mathbf{k}}$ given as the image of $r\mapsto (r,\mathbf{k})$.  Since $2^J \approx h^{1/2}$, we have that $\widetilde{\mathcal{T}}_\mathbf{k}\subset \mathcal{T}_{C_1h^{1/2}}(\gamma_{\mathbf{k}})$ for some $C_1$ large, and \eqref{finitecover} follows easily in any coordinate system.  We now take $\mathbf{k}_0=0$ and it is then verified that we may take $C$ so that \eqref{cardinality} is satisfied and $ d(\widetilde{\mathcal{T}}_\mathbf{k},\widetilde{\mathcal{T}}_{\mathbf{k}_0}) \gtrsim |\mathbf{k}-\mathbf{k}_0|$  whenever $d\left(\widetilde{\mathcal{T}}_\mathbf{k}, \mathscr{E}_\mu^j \right) \geq C2^j$.  We now revert back to the original coordinates and at the cost of enlarging $C$ and the other implicit constants, \eqref{distant},  \eqref{cardinality}, and \eqref{mutualsep} are all satisfied here since the diffeomorphism can be taken to satisfy a Lipschitz bound with uniform constant.

Given the above, \eqref{finitecover} means that it suffices to show the two inequalities
\begin{gather}
  \|B_{\mu,\omega}^j(s)f_h(s)\|_{L^2(\RR^{d-1})}^2 \lesssim
  (h^{-1/2}2^j)^{d-1} \sup_{\mathbf{k}} \left\|(\chi_\mathbf{k}f_h)(s,\cdot)\right\|_{L^2(\RR^{d-1})}^2, \label{firstMKN} \\
      \int_{-\veps}^{\veps}\|(\chi_\mathbf{k}f_h)(s,\cdot)\|_{L^2(\RR^{d-1})}^2 ds \lesssim  \sup_{\mathbf{\tilde k}} \int\left\|(\chi_\mathbf{\tilde k}\psi_h)(t,\cdot)\right\|_{L^2(\RR^{d-1})}^2 dt, \label{secondMKN}
\end{gather}
We primarily focus on \eqref{firstMKN}, as the second bound will be seen to follow from very similar ideas.  Let $B_{\mu,\omega,\mathbf{k}}(s)$ be the operator with Schwartz kernel given by
\begin{equation}\label{kernel}
\frac{1}{(2\pi h)^{d-1}} \int e^{\frac ih \langle y-x,\eta\rangle}b_{\mu,\omega}(s,y,\eta)\chi_{\mathbf{k}}(s,x)d\eta,
\end{equation}
and since $\sum_{\mathbf{k}}\chi_k^2 =1$, this means that
\begin{equation*}
B_{\mu,\omega}^j(s)(f_h(s,\cdot)) = B_{\mu,\omega}^j(s)\left(\sum_{\mathbf{k}} (\chi^2_k f_h)(s,\cdot)\right) = \sum_{\mathbf{k}} B_{\mu,\omega,\mathbf{k}}(s)\left((\chi_{\mathbf{k}} f_h)(s,\cdot)\right).
\end{equation*}
Partition the indices $\mathbf{k}$ into sets $K_1,K_2$ where $\mathbf{k} \in K_1$ if  $d\left(\widetilde{\mathcal{T}}_\mathbf{k}, \mathscr{E}_\mu^j \right) < C2^j$ and $\mathbf{k} \in K_2$ otherwise.  The compound symbol of $B_{\mu,\omega,\mathbf{k}}(s)$ is in $S_{1/2}(1)$ so the operator is uniformly bounded in $L^2(\RR^{d-1})$. Hence
\begin{align*}
\left\|B_{\mu,\omega}(s)\left(  \sum_{\mathbf{k}\in K_1} \chi_{\mathbf{k}} f_h (s,\cdot)\right) \right\|_{L^2(\RR^{d-1})}^2 &\lesssim \int \Big| \sum_{k\in K_1} \chi_{\mathbf{k}}^2(x) f_h (s,x)\Big|^2 dx \\
&\lesssim \sum_{k\in K_1} \int \big| \chi_{\mathbf{k}}(x) f_h (s,x)\big|^2 dx .
\end{align*}
where we have used \eqref{mutualsep} and $\chi_{\mathbf{k}}^2(x) \leq \chi_{\mathbf{k}}(x) $ as a bump function.  Given \eqref{cardinality}, the expression on the right is dominated by the right hand side of \eqref{firstMKN}.

We conclude \eqref{firstMKN} by showing that for any $N \geq 0$,
\begin{equation}\label{tripleBbounds}
\left\|B_{\mu,\omega,\mathbf{k}}(s)\right\|_{L^2 \to L^2} \lesssim_N (h^{-\frac 12}|\mathbf{k}-\mathbf{k}_0| )^{-N}
\end{equation}
so that by taking $N$ large
\begin{equation}\label{rapiddk}
\left\|\sum_{\mathbf{k}\in K_2} B_{\mu,\omega,\mathbf{k}}(s)\left((\chi_{\mathbf{k}} f_h)(s,\cdot)\right)\right\|_{L^2} \lesssim \sum_{\mathbf{k}\in K_2}(2^{-J}|\mathbf{k}-\mathbf{k}_0| )^{-N} \left\| (\chi_{\mathbf{k}} f_h)(s,\cdot)\right\|_{L^2}
\end{equation}
and the last expression is seen to be bounded by the right hand side of \eqref{firstMKN}.  To see \eqref{tripleBbounds}, observe that an integration by parts in \eqref{kernel} similar to \eqref{firstorderdo} shows that we may take the compound symbol $b_{\mu,\omega,\mathbf{k}}^j(s,y,x,\eta)$ of $B_{\mu,\omega,\mathbf{k}}^j(s)$ to satisfy
\begin{equation*}
  |(h^{\frac 12}\prtl_{y,x,\eta})^\alpha b_{\mu,\omega,\mathbf{k}}^j(s,\cdot)| \lesssim_{\alpha,N} (1+h^{-1}|y-x|^2)^{-N}
  \left(1+h^{-\frac 12}d\left(y,\eta;\kappa_{s,0}\left(\mathscr{D}_{\mu}^j\right)\right)\right)^{-2N}.
\end{equation*}
The right hand side here is thus $O((1+h^{-1/2}d(x,\mathscr{E}_\mu^j))^{-N})$ and since we may restrict to $x$ such that $\chi_{\mathbf{k}}(s,x)\neq 0$, this in turn is $O((h^{-\frac 12}|\mathbf{k}-\mathbf{k}_0| )^{-N})$ by \eqref{distant}.

The proof of \eqref{secondMKN} is similar, the primary difference is that $f_h$ is the image of $\psi_h$ under the operator in \eqref{fhoperator}, so the estimates are over $\RR^d$ instead of $\RR^{d-1}$.  But given \eqref{mutualsep}, the same principles as before apply here, as the kernel of the operator in \eqref{fhoperator} rapidly decreases on the scale of $h^{1/2}$ away from the diagonal.  The only other significant difference is that the analog of \eqref{rapiddk} should be adjusted to read
\begin{equation*}
  \left\|\sum_{\tilde{\mathbf{k}}\in K_2} \chi_{\tilde{\mathbf{k}}} f_h\right\|_{L^2(\RR^d)} \lesssim \sum_{\tilde{\mathbf{k}}\in K_2}(2^{-J}|\mathbf{k}-\tilde{\mathbf{k}}| )^{-N} \left\| \chi_{\tilde{\mathbf{k}}} \psi_h \right\|_{L^2(\RR^d)}.
\end{equation*}
\end{proof}
We now return to estimating \eqref{whitneyorgnorm} and state the two main theorems which form the crux of
the proof of Theorem \ref{thm:mainthm}.  The first step in is to use an almost orthogonality theorem, valid when $j=J$ or when $j \geq J+1$, which will be proved in \S\ref{sec:ao}. Note that $\kappa_{t,0}(\mu) = (x_{t,0}(\mu),\xi_{t,0}(\mu))$ is the image of $\mu \in 2^j \mathbb{Z}^{2(d-1)}$ under the map $\kappa_{t,0}$ defined after \eqref{duhamel}.
\begin{theorem}\label{thm:aothm}
Suppose $1 \leq p \leq \infty$ and $p^* = \min(p,p')$.  There exists a family of $h$-pseudodifferential operators $B_{\mu,\omega}^j(s)$ which satisfy \eqref{actualdecay} and with $\omega$ ranging over a finite collection of indices depending only on the dimension, such that
\begin{multline}\label{aoj}
 \left\|\sum_{\mu, \mu' \in \Xi_j} T(B_{\mu}^jf_h) T(B_{\mu'}^jf_h)\right\|_{L^{p}} \lesssim\\
\sum_{\omega,\omega'}\left(\sum_{\mu, \mu' \in \Xi_j} \left\|T_{\mu,\omega}^j (B_{\mu,\omega}^jf_h)  T_{\mu',\omega'}(B_{\mu',\omega'}^jf_h)\right\|_{L^{p}}^{p^*}\right)^{\frac{1}{p^*}},
\end{multline}
when $0 \geq j \geq J+1$ or when $j=J$.  Here $T_{\mu,\omega}^j$ (and similarly $T_{\mu',\omega'}^j$) is defined by
\begin{equation}\label{tmuomega}
(T_{\mu,\omega}^jF) (t,x) =\int_{-\veps}^t \left(T_{\mu,\omega}^{j,s}F(s,\cdot)\right)(x) \,ds
\end{equation}
where $T_{\mu,\omega}^{j,s}$ is the fixed $s$ operator defined by the formula
\begin{equation}\label{tmuomegafixeds}
(T_{\mu,\omega}^{j,s}f)(t,x)=\frac{1}{(2\pi h)^{d-1}} \int e^{\frac ih (\phi(t,s,x,\eta)-\langle y ,\eta \rangle)} a^j_{\mu,\omega}(t,s,x,\eta) f(y) \,dyd\eta
\end{equation}
with $\phi$ as in \eqref{parametrix} but with symbol $a^j_{\mu,\omega}(t,s,x,\eta)$ satisfying
\[
|\prtl^\alpha_{x,\eta} a^j_{\mu,\omega}(t,s,x,\eta)|\lesssim 2^{-j|\alpha|}(1+2^{-j}|x-x_{t,0}(\mu)|)^{-2d}(1+2^{-j}|\eta-\xi_{s,0}(\mu)|)^{-2d}.
\]
\end{theorem}

The second step in estimating this is to employ the bilinear estimates in the following theorem, which we prove in \S\ref{sec:bilinear}:
\begin{theorem}\label{thm:keybilinear}
Suppose $j=J,\dots,0$, $\mu,\mu' \in \Xi_j$ and  $T_{\mu,\omega}^j, T_{\mu',\omega'}^j,B_{\mu,\omega}^j,B_{\mu',\omega'}^j$ are as in Theorem \ref{thm:aothm}.  There exist corresponding ($s$ dependent) $h$-PDO $\tilde{B}_{\mu,\omega}^j$ and $\tilde{B}_{\mu',\omega'}^j$ satisfying \eqref{actualdecay} such that
the bilinear operator $R$ defined for fixed $s,s'\in[-\veps,\veps]$ by
\[
R(B_{\mu,\omega}^j(s)f,B_{\mu',\omega'}^j(s')g)(t,x) := \left(T_{\mu,\omega}^{j,s} (B_{\mu,\omega}^j(s)f)\right)(t,x) \cdot \left(T_{\mu',\omega'}^{j,s'} (B_{\mu',\omega'}^j(s')g)\right)(t,x).
\]
satisfies bounds
\begin{multline}\label{keybilinear}
h^{-\frac{2d}{q}+(d-1)}2^{-j(d-1-\frac{2(d+1)}{q})}
\left\|R(B_{\mu,\omega}^j(s)f,B_{\mu',\omega'}^j(s')g)\right\|_{L^{\frac q2}([-\veps,\veps]_t \times \RR^{d-1}_x)}
\lesssim\\
\|B_{\mu,\omega}^j(s)f\|_{L^2}\|B_{\mu',\omega'}^j(s')g\|_{L^2}+
\|\tilde{B}_{\mu,\omega}^j(s)f\|_{L^2}\|B_{\mu',\omega'}^j(s')g\|_{L^2}\\
+\|B_{\mu,\omega}^j(s)f\|_{L^2}\|\tilde{B}_{\mu',\omega'}^j(s')g\|_{L^2}+
\|\tilde{B}_{\mu,\omega}^j(s)f\|_{L^2}\|\tilde{B}_{\mu',\omega'}^j(s')g\|_{L^2},
\end{multline}
where the norms on the right hand side are taken to be $L^2(\RR^{d-1})$.
\end{theorem}

Note that the first factor on the left rewrites as $h^{(d-1)(\frac 12-\frac 1q)}(h^{\frac 12}2^{-j})^{\frac{d-1}{q}(q-\frac{2(d+1)}{d-1})}$, showing that the phase space separation of the $B_{\mu,\omega}^j$ and $B_{\mu',\omega'}^j$ (up to rapidly decreasing tails) results in a gain of $(h^{\frac 12}2^{-j})^{\frac{d-1}{q}(\frac{2(d+1)}{d-1}-q)}$ over what would be obtained by applying the linear estimates \eqref{linearbounds}. The $\tilde{B}_{\mu,\omega}^j$ are slight distortions of the $B_{\mu,\omega}^j$ that arise when replacing the rapidly decaying symbol of $T^j_{\mu,\omega}$ by a compactly supported one.  We take the family \eqref{Bfamily} to consist of the union of the operators constructed in these two theorems, that is, operators of the form  $B_{\mu,\omega}^j(s)$ or $\tilde{B}_{\mu,\omega}^j(s)$.

We can now begin discussing how to estimate the terms in \eqref{whitneyorgnorm}, showing that the two theorems here imply \eqref{higherdimbound}, \eqref{twodimboundL4}, and hence Theorem \ref{thm:mainthm}.  Given Theorem \ref{thm:aothm}, we have for any $J \leq j \leq 0$,
\begin{multline*}
\left\|\sum_{(\mu, \mu') \in \Xi_j } T(B_{\mu}^jf_h) T(B_{\mu'}^jf_h)\right\|_{L^{\frac q2}}\\ \lesssim
\sum_{\omega,\omega'}\left(\sum_{\mu, \mu' \in \Xi_j} \left\|T_{\mu,\omega}^j (B_{\mu,\omega}^jf_h)  T_{\mu',\omega'}^j (B_{\mu',\omega'}^jf_h) \right\|_{L^{\frac q2}}^{(\frac{q}{2})^*}\right)^{\frac{1}{(\frac{q}{2})^*}}.
\end{multline*}
where the $L^{\frac q2}$ norms are taken over $[-\veps,\veps]_t \times \mathbb{R}^{d-1}_x$. We now fix $\omega,\omega'$ and show that Theorem \ref{thm:keybilinear} supplies bounds on each term on the right hand side here.  In this case Minkowski's inequality for integrals shows that
\begin{multline*}
\left\|T_{\mu,\omega}^j (B_{\mu,\omega}^jf_h)  T_{\mu',\omega'}^j (B_{\mu',\omega'}^jf_h)\right\|_{L^{\frac q2}}
\lesssim\\
\int_{[-\veps,\veps]^2} \left(\int_{-\veps}^{\veps}\left\| T^{j,s}_{\mu,\omega}(B_{\mu,\omega}^j(s)f_h(s)) \cdot T^{j,s'}_{\mu',\omega'}(B_{\mu',\omega'}^j(s')f_h(s'))\right\|_{L^{\frac q2}(\RR^{d-1})}^{\frac q2} dt\right)^{\frac 2q}\,ds ds'
\end{multline*}
Hence \eqref{keybilinear} and H\"older's inequality show that this in turn is bounded by
\begin{multline}\label{messyrhs}
h^{\frac{2d}{q}-(d-1)}2^{j(d-1-\frac{2(d+1)}{q})} \Big( \|B_{\mu,\omega}^jf_h\|_{L^2(\RR^{d})}\|B_{\mu',\omega'}^jf_h\|_{L^2(\RR^{d})}+\\
\|\tilde{B}_{\mu,\omega,}^jf_h\|_{L^2(\RR^{d})}\|B_{\mu',\omega'}^jf_h\|_{L^2(\RR^{d})}+
\|B_{\mu,\omega}^jf_h\|_{L^2(\RR^{d})}\|\tilde{B}_{\mu',\omega'}^jf_h\|_{L^2(\RR^{d})}+\\
\|\tilde{B}_{\mu,\omega}^jf_h\|_{L^2(\RR^{d})}\|\tilde{B}_{\mu',\omega'}^jf_h\|_{L^2(\RR^{d})}\Big) .
\end{multline}

We will now bound the right hand side of \eqref{messyrhs} in two different ways, and then show how to optimize the choice.  If $(\mu,\mu') \in \Xi_j$, then $2^{-j}(\mu-\mu')$ lies in a fixed, finite collection of vectors (of size $\approx 1$ when $j \geq J+1$ or $O(1)$ when $j=J$).  Hence
\begin{equation}\label{cardinalXi}
 \#\{\mu': (\mu,\mu') \in \Xi_j \} = O(1) \text{ for any fixed } \mu \in 2^{j}\mathbb{Z}^{2(d-1)}.
\end{equation}
We now have for any two $\omega, \omega'\in \Omega$,
\begin{multline}\label{messyaorhs}
\left\|T_{\mu,\omega}^j f_h  T_{\mu',\omega'}^jf_h\right\|_{\ell^{(\frac q2)^*}_\mu L^{\frac q2}} \lesssim \\ h^{\frac{2d}{q}-(d-1)}2^{j(d-1-\frac{2(d+1)}{q})}  \left( \sum_{\mu} \sum_{\mu': (\mu,\mu') \in \Xi_j }\|B_{\mu,\omega}^jf_h\|_{L^2}^{(\frac{q}{2})^*}\|B_{\mu',\omega'}^jf_h\|_{L^2}^{(\frac{q}{2})^*}\right)^{\frac{1}{(\frac{q}{2})^*}}+ \cdots
\end{multline}
where the dots denote the corresponding contributions of the last three terms in \eqref{messyrhs}.  Observe that when $d \geq 3$, $(\frac q2)^*=\frac q2$ (since $q<\frac{2(d+1)}{d-1} \leq 4$ in this case) and hence by Cauchy-Schwarz and \eqref{cardinalXi}, the right hand side in \eqref{messyaorhs} is dominated by
\[
h^{\frac{2d}{q}-(d-1)}2^{j(d-1-\frac{2(d+1)}{q})}  \left( \sum_{\mu} \|B_{\mu,\omega}^jf_h\|_{L^2}^{q}\right)^{\frac{1}{q}}  \left( \sum_{\mu'}
\|B_{\mu',\omega'}^jf_h\|_{L^2}^{q}\right)^{\frac{1}{q}} + \cdots
\]
Indeed given \eqref{cardinalXi}, we may view $\mu'$ as a function of $\mu$ at the cost of increasing the implicit constant by a factor determined the $O(1)$ constant there.  By the embedding $\ell^2 \hookrightarrow \ell^q$ and Lemma \ref{lem:aosums}, we have that this in turn is bounded by
$
h^{\frac{2d}{q}-(d-1)}2^{j(d-1-\frac{2(d+1)}{q})}  \|f_h\|_{L^2}^{2}
$.  When $d=2$, the same conclusion holds since $(\frac q2)^*=(\frac q2)'=\frac{q}{q-2}$ and $\ell^{2} \hookrightarrow \ell^{\frac{2q}{q-2}}$.

On the other hand, when $d\geq 3$ we may estimate the sum appearing on the right in \eqref{messyaorhs} (and similarly the other terms) by
\begin{multline*}
\left( \sum_{\mu} \sum_{\mu': (\mu,\mu') \in \Xi_j } \|B_{\mu,\omega}^jf_h\|_{L^2}^{\frac{q}{2}} \|B_{\mu',\omega'}^jf_h\|_{L^2}^{\frac{q}{2}}\right)^{\frac{2}{q}}
 \lesssim \|f_h\|_{L^2} \left(\sum_{\mu'} \|B_{\mu',\omega'}^jf_h\|_{L^2}^{\frac{2q}{4-q}}\right)^{\frac{4-q}{2q}}\\
\lesssim \|f_h\|_{L^2}^{\frac 4q}\left(\sup_{\mu'} \|B_{\mu',\omega'}^jf_h\|_{L^2} \right)^{2-\frac 4q} \lesssim 2^{j(d-1)(1-\frac 2q)}\|f_h\|_{L^2}^{\frac 4q} \|f_h\|_{MKN}^{2-\frac 4q},
\end{multline*}
where the first inequality follows by H\"older's inequality with exponents satisfying $\frac 2q = \frac 12 +\frac{4-q}{2q}$ and the second inequality follows by using that $\frac{2q}{4-q}=2+\frac{4q-8}{4-q}$.  Note that we have used Lemma \ref{lem:aosums} twice in the process and also \eqref{cardinalXi} similar to before.  When $d=2$, a similar argument checks that the sum appearing on the right in \eqref{messyaorhs} is majorized by
\[
\|f_h\|_{L^2}^{2-\frac 4q}\left(\sup_{\mu'} \|B_{\mu',\omega'}^jf_h\|_{L^2} \right)^{\frac 4q}\lesssim 2^{\frac{2j}{q}}  \|f_h\|_{L^2}^{2-\frac 4q}\|f_h\|_{MKN}^{\frac 4q}.
\]

Repeating this for the finite collection of pairs $\omega,\omega'$, we have that when $d\geq 3$,
\begin{equation}\label{MKNbound}
\left\|\sum_{(\mu, \mu') \in \Xi_j } T(B_{\mu}^jf_h) T(B_{\mu'}^jf_h)\right\|_{L^{\frac q2}} \lesssim h^{\frac{2d}{q}-(d-1)}2^{j((d-1)(2-\frac 2q)-\frac{2(d+1)}{q})}  \|f_h\|_{L^2}^{\frac 4q}\|f_h\|_{MKN}^{2-\frac 4q},
\end{equation}
though the second factor on the right rewrites as $2^{2j(\frac{d-1}{q})(q-\frac{2d}{d-1})}$.  When $d=2$, \begin{equation}\label{MKNbound2d}
\left\|\sum_{(\mu, \mu') \in \Xi_j } T(B_{\mu}^jf_h) T(B_{\mu'}^jf_h)\right\|_{L^{\frac q2}} \lesssim h^{\frac{2d}{q}-(d-1)}2^{j(1-\frac 4q)}  \|f_h\|_{L^2}^{2-\frac 4q}\|f_h\|_{MKN}^{\frac 4q}.
\end{equation}
At the same time, in all cases we have
\begin{equation}\label{CSbound}
\left\|\sum_{(\mu, \mu') \in \Xi_j } T(B_{\mu}^jf_h) T(B_{\mu'}^jf_h)\right\|_{L^{\frac q2}} \lesssim h^{\frac{2d}{q}-(d-1)}2^{j(d-1-\frac{2(d+1)}{q})}  \|f_h\|_{L^2}^{2}.
\end{equation}

Consequently, if we fix an integer $J \leq L \leq 0$, then when $d \geq 3$, \eqref{MKNbound} yields
\begin{multline*}
\sum_{j=J}^L\left\|\sum_{(\mu, \mu') \in \Xi_j } T(B_{\mu}^jf_h) T(B_{\mu'}^jf_h)\right\|_{L^{\frac q2}} \lesssim\\
 h^{\frac{2d}{q}-(d-1)}2^{L((d-1)(2-\frac 2q)-\frac{2(d+1)}{q})} \|f_h\|_{L^2}^{\frac 4q}\|f_h\|_{MKN}^{2-\frac 4q},
\end{multline*}
since $q>\frac{2(d+2)}{d}>\frac{2d}{d-1}$ implies that $(d-1)(2-\frac 2q)-\frac{2(d+1)}{q}>0$.  Moreover, $d-1-\frac{2(d+1)}{q}<0$, so \eqref{CSbound} yields
\[
\sum_{j=L}^0\left\|\sum_{(\mu, \mu') \in \Xi_j } T(B_{\mu}^jf_h) T(B_{\mu'}^jf_h)\right\|_{L^{\frac q2}} \lesssim h^{\frac{2d}{q}-(d-1)}2^{L(d-1-\frac{2(d+1)}{q})}  \|f_h\|_{L^2}^{2}.
\]
It is then verified that we optimize the bounds on \eqref{whitneyorgnorm} by taking $L$ such that $2^{L(d-1)/2}\approx \|f_h\|_{L^2}/\|f_h\|_{MKN}$, which is possible since the latter quantity is $O(1)$ and bounded below by $h^{\frac{d-1}{4}}$. This shows that \eqref{whitneyorgnorm} is bounded by
\[
h^{\frac{2d}{q}-(d-1)} \|f_h\|_{L^2}^{4-\frac{4(d+1)}{q(d-1)}}\|f_h\|_{MKN}^{\frac{4(d+1)}{q(d-1)}-2}
\]
When $d=2$ and $4<q<6$, a similar argument using \eqref{MKNbound2d} and \eqref{CSbound} shows the same bound. Finally, when $d=2$ and $q=4$ we obtain that \eqref{whitneyorgnorm} is bounded by
\[
\log(h^{-1})\|f_h\|_{L^2}\|f_h\|_{MKN}.
\]
Given \eqref{fhbound} and Lemma \ref{lem:MKNandKN}, this completes the proofs of \eqref{higherdimbound}, \eqref{twodimboundL4}.

Another perspective on \eqref{MKNbound} and \eqref{CSbound} (and similarly when $d=2$) results from noting that right hand sides of these two inequalities are respectively bounded by
\begin{equation*}
\begin{gathered}
  h^{-2\delta(q)}(h2^{-2L})^{\frac{d-1}{q}(q-\frac{2d}{d-1})}\|\psi_h\|_{L^2}^{\frac 4q}|||\psi_h|||_{KN}^{2-\frac 4q},\\
  h^{-2\delta(q)}(h2^{-2L})^{\frac{d-1}{2q}(\frac{2(d+1)}{d-1}-q)}\|\psi_h\|_{L^2}^2,
\end{gathered}
\end{equation*}
after an application of \eqref{MKNandKN} and $\|f_h\|_{L^2} \lesssim \|\psi_h\|_{L^2}$ (as well as recalling the differences between \eqref{KNnorm}, \eqref{KNnormNonavg}).  The latter of these underscores the gain resulting from the bilinear estimates while the former trades a loss in $h2^{-2L}$ for a gain in the Kakeya-Nikodym norm of $\psi_h$.  This also shows that $L$ is taken so that
\begin{equation*}
  |||\psi_h|||_{KN}\approx \|\psi_h\|_{L^2}(h2^{-2L})^{\frac{d-1}{4}}.
\end{equation*}

\section{Almost Orthogonality}\label{sec:ao}
In this section, we prove Theorem \ref{thm:aothm}.  We begin with a lemma, recalling the notation for $(x_{t,0}(\mu),\xi_{t,0}(\mu))$ following \eqref{duhamel} and prior to Theorem \ref{thm:aothm}.
\begin{lemma}\label{lem:aoprep}
Let $2^j \in [h^{\frac 12},1]$, $\mu \in 2^j\mathbb{Z}^{2(d-1)}$ and suppose that $\alpha,\beta$ are multiindices such that $|\alpha|,|\beta|\leq 2d$.  Given $f \in \mathcal{S}(\RR^{d-1})$, let
\[
F_\mu^s(y) := \langle 2^{-j}(y-x_{s,0}(\mu)) \rangle^{4d}(\langle 2^{-j}(hD_y-\xi_{s,0}(\mu))\rangle^{4d} f)(y),
\]
with $\langle w \rangle = (1+|w|^2)^{\frac 12}$. There exists a symbol $a^{j}_{\mu,\alpha,\beta}(t,s,x,\xi)$ satisfying
\begin{equation}\label{qmusymbolbounds}
 \left|D^{\gamma}_{x,\xi}a^{j}_{\mu,\alpha,\beta}(t,s,x,\xi)\right| \lesssim_{\gamma} 2^{-j|\gamma|}\langle 2^{-j}(y-x_{t,0}(\mu)) \rangle^{-2d}\langle 2^{-j}(\xi-\xi_{s,0}(\mu))\rangle^{-2d}
\end{equation}
such that up to negligible $O(h^\infty)$ error,
\begin{multline}
2^{-j(|\alpha|+|\beta|)}(x-x_{t,0}(\mu))^{\alpha}(hD_x-\xi_{t,0}(\mu))^{\beta} (S(t,s)f)(x) = \\
\frac{1}{(2\pi h)^{d-1}} \iint e^{\frac ih (\phi(t,s,x,\xi)-\langle y ,\xi\rangle)} a^{j}_{\mu,\alpha,\beta}(t,s,x,\xi) F_\mu^s(y)\,dy d\xi.
\end{multline}
\end{lemma}
\begin{proof}
Let $A(t,s,x,y,\eta)$ be defined by
\[
 A(t,s,x,y,\eta) = a(t,s,x,\eta)\langle 2^{-j}(\eta-\xi_{s,0}(\mu))\rangle^{-4d}
 \langle 2^{-j}(y-x_{s,0}(\mu))\rangle^{-4d}
\]
so that up to $O(h^\infty)$ error (resulting only from the error term in \eqref{parametrix}),
\begin{multline*}
(2\pi h)^{d-1}S(t,s)f(x)  \\
=\int e^{\frac ih \phi(t,s,x,\eta)} a(t,s,x,\eta)
\langle 2^{-j}(\eta-\xi_{s,0}(\mu))\rangle^{-4d} \mathscr{F}_h(\langle 2^{-j}(hD-\xi_{s,0}(\mu))\rangle^{-4d}f)(\eta) d\eta\\
=\iint e^{\frac ih (\phi(t,s,x,\eta)-\langle y ,\eta \rangle)} A(t,s,x,y,\eta) F_\mu^s(y)\,dy d\eta,
\end{multline*}
where we have inverted $\mathscr{F}_h$ in the last expression and used that the argument of $\mathscr{F}_h$ in the second expression is $\langle 2^{-j}(y-x_{s,0}(\mu))\rangle^{-4d}F_\mu^s(y)$.
Next we write
\begin{multline}\label{A1integral}
2^{-j(|\alpha|+|\beta|)}(x-x_{t,0}(\mu))^\alpha  (hD_x-\xi_{t,0}(\mu))^\beta \int e^{\frac ih (\phi(t,s,x,\eta)-\langle y ,\eta \rangle)} A(t,s,x,y,\eta) \,d\eta\\=
\int e^{\frac ih (\phi(t,s,x,\eta)-\langle y ,\eta \rangle)} A_1(t,s,x,y,\eta) \,d\eta
\end{multline}
where
\begin{multline*}
2^{j|\gamma|} |D^{\gamma}_{x,y,\eta}A_1(t,s,\cdot)| \lesssim_{\gamma}
\langle 2^{-j}(\eta-\xi_{s,0}(\mu))\rangle^{-4d}
\langle 2^{-j}(y-x_{s,0}(\mu))\rangle^{-4d}
\\
\times \langle 2^{-j}(x-x_{t,0}(\mu))\rangle^{|\alpha|}
\langle 2^{-j}(d_x\phi(t,s,x,\eta)-\xi_{t,0}(\mu))\rangle^{|\beta|}.
\end{multline*}

The differential operator
\begin{equation}\label{phasepreserveA1}
\frac{1+2^{-j}(d_\eta\phi(t,s,x,\eta)-y)\cdot (h2^{-j}D_\eta)}{1+2^{-2j}|d_\eta\phi(t,s,x,\eta)-y|^2}
\end{equation}
preserves the exponential function in the integral on the right in \eqref{A1integral}.  Moreover,  since $h 2^{-j} \leq 2^j$, the weighted derivative $h2^{-j}D_\eta$ appearing here means that the adjoint of the differential operator applied to $A_1$ has the effect of replacing that amplitude by one of the same regularity.  Using that
\[
\langle 2^{-j}|y-x_{s,0}(\mu)|\rangle^{-4d}
\langle 2^{-j}|d_\eta\phi(t,s,x,\eta)-y|\rangle^{-4d}
\lesssim
\langle 2^{-j}|d_\eta\phi(t,s,x,\eta)-x_{s,0}(\mu)|\rangle^{-4d},
\]
integration by parts shows that \eqref{A1integral} can further be rewritten as
\[
\int e^{\frac ih (\phi(t,s,x,\eta)-\langle y ,\eta \rangle)} A_2(t,s,x,y,\eta) \,d\eta
\]
where
\begin{multline*}
2^{j|\gamma|} |D^{\gamma}_{x,y,\eta}A_2(t,s,\cdot)| \lesssim_{\gamma} \\
\langle 2^{-j}(\eta-\xi_{s,0}(\mu))\rangle^{-4d}
\langle 2^{-j}(d_\eta\phi(t,s,x,\eta)-x_{s,0}(\mu))\rangle^{-4d}
\\
\times \langle 2^{-j}(x-x_{t,0}(\mu))\rangle^{|\alpha|}
\langle 2^{-j}(d_x\phi(t,s,x,\eta)-\xi_{t,0}(\mu))\rangle^{|\beta|}
\\
\lesssim \langle 2^{-j}(\eta-\xi_{s,0}(\mu))\rangle^{-2d}
\langle 2^{-j}(x-x_{t,0}(\mu))\rangle^{-2d},
\end{multline*}
where the last inequality follows since $|\alpha|, |\beta| \leq 2d$ and $\kappa_{t,s}$ is Lipschitz in $(x,\xi)$.

We now define
\begin{equation}\label{A2integral}
a^j_{\mu,\alpha,\beta}(t,s,x,\xi) := \frac{e^{-\frac ih \phi(t,s,x,\xi)}}{(2\pi h)^{d-1}}\iint e^{\frac ih \left( \phi(t,s,x,\eta)+\langle y,\xi-\eta\rangle \right)}
 A_2(t,s,x,y,\eta)\,dyd\eta
\end{equation}
and are left to show \eqref{qmusymbolbounds}.  Abbreviate the phase function here as
\[
 \Phi(t,s,x,y,\eta,\xi) = \phi(t,s,x,\eta)-\phi(t,s,x,\xi)+\langle y,\xi-\eta\rangle .
\]
First observe that for a single derivative
\[
D_{x_m} e^{\frac ih \Phi} =
-\sum_{l=1}^{d-1} \left(\int_0^1\prtl_{x_m \xi_l}^2 \phi(t,s,x,(1-r)\xi+r\eta)\,dr\right)  D_{y_l} e^{\frac ih \Phi}.
\]
Integration by parts thus shows that applying a weighted derivative $2^jD_{x_m}$ to $a^j_{\mu,\alpha,\beta}$ has the effect of replacing $A_2$ by an amplitude of the same regularity.  For derivatives in $\xi$, note that
\begin{equation*}
 D_{\xi_m} e^{\frac ih \Phi} = h^{-1}\left(\prtl_{\eta_m} \phi(t,s,x,\eta) - \prtl_{\xi_m} \phi(t,s,x,\xi)\right) e^{\frac ih \Phi} - D_{\eta_m}e^{\frac ih \Phi}.
\end{equation*}
As before, the first term on the right here is thus equivalent to the action of a vector field acting on the exponential in the $y$ variable.  Hence integration by parts in $y,\eta$ similar to before shows that applying a weighted derivative $2^jD_{\xi_m}$ to $a^j_{\mu,\alpha,\beta}$ again has the effect of replacing $A_2$ by an amplitude of the same regularity.

We now conclude the proof by observing that the expression in \eqref{A2integral} is uniformly bounded in $h$.  But this follows from integrating by parts with respect to the operator in \eqref{phasepreserveA1} with $2^{-j} = h^{-1/2}$ sufficiently many times so that integration in $y,\eta$ yields a gain of $O(h^{d-1})$.
\end{proof}
\begin{proof}[Proof of Theorem \ref{thm:aothm}]  Let $Q_{1,\mu}^j(t)$, $Q_{2,\mu}^j(t)$ be the multiplier and Fourier multiplier
\[
Q_{1,\mu}^j(t) := (1+ 2^{-2j}|x-x_{t,0}(\mu)|^2)^{-2d}, \quad
Q_{2,\mu}^j(t) := (1+ 2^{-2j}|hD_x-\xi_{t,0}(\mu)|^2)^{-2d}.
\]
Thus if $c(w) $ denotes the polynomial $c(w) = (1+ 2^{-2j}|w|^2)^{2d}$ for $w \in \RR^{d-1}$, we have
\begin{multline}\label{productinvert}
T(B_{\mu}^jf_h)\cdot T(B_{\mu'}^jf_h) = \\
Q_{1,\mu}^j(t) Q_{2,\mu}^j(t) c(x-x_{t,0}(\mu))c(hD_x-\xi_{t,0}(\mu))\left(T(B_{\mu}^jf_h)(t)\cdot T(B_{\mu'}^jf_h)(t)\right)  \\
+Q_{1,\mu}^j(t) [c(x-x_{t,0}(\mu)), Q_{2,\mu}^j(t)] c(hD_x-\xi_{t,0}(\mu))\left(T(B_{\mu}^jf_h)(t)\cdot T(B_{\mu'}^jf_h)(t)\right).
\end{multline}
In turn the commutator $[c(x-x_{t,0}(\mu)), Q_{2,\mu}^j(t)] $ can be realized as
\begin{equation}\label{spatialcomm}
[c(x-x_{t,0}(\mu)), Q_{2,\mu}^j(t)]=\sum_{|\alpha|\leq 2d} Q_{2,\mu}^{j,\alpha}(t) c_\alpha (x-x_{t,0}(\mu))
\end{equation}
where the symbol of $Q_{2,\mu}^{j,\alpha}(t)$ is $(h2^{-j}\prtl_\xi)^\alpha(1+ 2^{-2j}|\xi-\xi_{t,0}(\mu)|^2)^{-2d}$ and $c_\alpha$ is a polynomial of degree $|\alpha|$.  We thus concern ourselves with the contribution of the middle line in \eqref{productinvert}, as similar arguments treat the last line.

We now observe that
\begin{equation}\label{polyproduct}
c(x-x_{t,0}(\mu))c(hD_x-\xi_{t,0}(\mu))\left(T(B_{\mu}^jf_h) T(B_{\mu'}^jf_h)\right)
\end{equation}
is a sum consisting of terms of the form
\[
2^{-j(|\alpha| + |\beta|+ |\gamma|)}(x-x_{t,0}(\mu))^\alpha (hD_x-\xi_{t,0}(\mu))^\beta T(B_{\mu}^jf_h)\cdot  (hD_x-\xi_{t,0}(\mu))^{\gamma} T(B_{\mu'}^jf_h).
\]
The operator $2^{-j|\gamma|}(hD_x-\xi_{t,0}(\mu))^{\gamma} $ can be written as a sum consisting of terms
\[
2^{-j|\gamma_1|}(\xi_{t,0}(\mu)-\xi_{t,0}(\mu'))^{\gamma_1}\cdot 2^{-j|\gamma_2|}(hD_x-\xi_{t,0}(\mu))^{\gamma_2} , \qquad\gamma_1+\gamma_2 = \gamma,
\]
but since $|\xi_{t,0}(\mu)-\xi_{t,0}(\mu')|\lesssim 2^{j}$, the first factor here is a bounded function.

We next claim that
\[
2^{-j(|\alpha| + |\beta|)}(x-x_{t,0}(\mu))^\alpha (hD_x-\xi_{t,0}(\mu))^\beta B_{\mu}^j f_h = B_{\mu,\alpha, \beta}^j f_h
\]
for some $h$-PDO $B_{\mu,\alpha, \beta}^j$ whose symbol satisfies \eqref{actualdecay}.  Its symbol is
\begin{equation*}
  2^{-j(|\alpha| + |\beta|)}(x-x_{t,0}(\mu))^\alpha (\xi-\xi_{t,0}(\mu))^\beta b_\mu^j(t,x,\xi).
\end{equation*}
For each $(x,\xi)$, we claim that the absolute value of this symbol is bounded above by $(1+h^{-\frac 12}d(x,\xi;\kappa_{t,0}(\mathscr{D}_{\mu}^j))^{-N}$. Take $(\tilde x,\tilde \xi) \in \kappa_{t,0}(\mathscr{D}_{\mu}^j)$ so that
\begin{equation*}
|(x,\xi) - (\tilde x,\tilde \xi)| \leq 2d(x,\xi;\kappa_{t,0}(\mathscr{D}_{\mu}^j))
\end{equation*}
and by a multinomial expansion the symbol rewrites as
\begin{equation*}
    2^{-j(|\alpha| + |\beta|)}\sum_{\gamma_1 \leq \alpha, \gamma_2\leq \beta} c_{\gamma_1,\gamma_2} (x-\tilde{x})^{\gamma_1}(\tilde x -x_{t,0}(\mu))^{\alpha-\gamma_1} (\xi-\tilde{\xi})^{\gamma_2}(\tilde \xi-\xi_{t,0}(\mu))^{\beta -\gamma_2} b_\mu^j,
\end{equation*}
for some coefficients $c_{\gamma_1,\gamma_2}$.  Since $|\tilde x -x_{t,0}(\mu) |, |\tilde \xi-\xi_{t,0}(\mu)| = O(2^j)$, and $2^{-j} \leq h^{-1/2}$ the absolute value of the symbol is dominated by
\begin{multline*}
    \sum_{\gamma_1 \leq \alpha, \gamma_2\leq \beta} 2^{-j(|\gamma_1| + |\gamma_2|)} |x-\tilde{x}|^{|\gamma_1|} |\xi-\tilde{\xi}|^{|\gamma_2|}\left(1+h^{-\frac 12}d\left(x,\xi;\kappa_{t,0}(\mathscr{D}_{\mu}^j\right)\right)^{-N-|\alpha|-|\beta|}\\
    \lesssim \left(1+h^{-\frac 12}d\left(x,\xi;\kappa_{t,0}(\mathscr{D}_{\mu}^j\right)\right)^{-N}.
\end{multline*}
Derivatives of the symbol of $B_{\mu,\alpha, \beta}^j$ are handled similarly as they will be a sum of terms involving derivatives of $b_\mu^j$ and monomials as above with smaller powers.

We now apply Lemma \ref{lem:aoprep}, and reindex the sum formed by \eqref{polyproduct} terms by $\omega, \omega'$, so that it writes as a sum
\[
\sum_{\omega,\omega'}T^j_{\mu,\omega}(B_{\mu,\omega}^j f_h) T^j_{\mu',\omega'}(B_{\mu',\omega'}^j f_h) ,
\]
where $T^j_{\mu,\omega}$ is of the form \eqref{tmuomega}, \eqref{tmuomegafixeds}.

Fix $\omega,\omega'$, and set $G_{\mu,\mu'} = T^j_{\mu,\omega}(B_{\mu,\omega}^j f_h) T^j_{\mu',\omega'}(B_{\mu',\omega'}^j f_h)$. We claim that
\begin{equation}\label{operatorao}
\left\|\sum_{\mu,\mu' \in \Xi_j} Q_{1,\mu}^j Q_{2,\mu}^j (G_{\mu,\mu'})\right\|_{L^p} \lesssim \left(\sum_{\mu,\mu' \in \Xi_j} \left\|G_{\mu,\mu'} \right\|_{L^p}^{p^*}\right)^{\frac{1}{p^*}},
\end{equation}
Observe that when $p=1$ or $p=\infty$, \eqref{operatorao} follows from the fact that $Q_{1,\mu}^j(t)$ is multiplication by a bounded function and the convolution kernel of $Q_{2,\mu}^j(t)$ is a function with uniform $L^1$ norm.  By interpolation, it suffices to prove \eqref{operatorao} when $p=2$.  Moreover, it is sufficient to prove it for $t$ fixed.

We thus write the left hand side of \eqref{operatorao} as
\[
\sum_{\substack{\mu,\mu' \in \Xi_j\\ \tilde{\mu},\tilde{\mu}' \in \Xi_j}}\int_{\RR^{d-1}} Q_{1,\mu}^j(t) Q_{2,\mu}^j(t) (G_{\mu,\mu'}(t,\cdot))
\overline{Q_{1,\tilde{\mu}}(t) Q_{2,\tilde{\mu}}(t) (G_{\tilde{\mu},\tilde{\mu}'}(t,\cdot))}.
\]
Given \eqref{cardinalXi}, the remainder of the proof for $p=2$ now follows by showing that the absolute value of each integral here is bounded by
\begin{equation}\label{museparation}
(1+2^{-2j}|\mu-\tilde{\mu}'|^2)^{-2d}\|G_{\mu,\mu'}(t,\cdot)\|_{L^2} \|G_{\tilde{\mu},\tilde{\mu}'}(t,\cdot)\|_{L^2} .
\end{equation}
Begin by writing $\mu=(\mu_x,\mu_\xi) \in \RR^{d-1} \times \RR^{d-1}$ and consider cases
\[
|\mu_x -\tilde{\mu}_x| \geq 8 |\mu_\xi-\tilde{\mu}_\xi| \text{ and }|\mu_x -\tilde{\mu}_x| < 8 |\mu_\xi-\tilde{\mu}_\xi|.\]  In the first case, we have $|x_{t,0}(\mu)-x_{t,0}(\mu')| \approx |\mu-\mu'|$  which shows that
\[
(1+ 2^{-2j}|x-x_{t,0}(\mu)|^2)^{-2d} (1+ 2^{-2j}|x-x_{t,0}(\tilde{\mu})|^2)^{-2d}\lesssim (1+ 2^{-2j}|\mu-\tilde{\mu}|^2)^{-2d},
\]
and hence \eqref{museparation} follows.  In the second case, $|\xi_{t,0}(\mu)-\xi_{t,0}(\mu')| \approx |\mu-\mu'|$ instead.  We thus use the Plancherel identity for the semiclassical Fourier transform and simple convolution estimates to see that
$Q_{1,\mu}^j(t) Q_{2,\mu}^j(t) (G_{\mu,\mu'}(t,\cdot))$ is sufficiently concentrated in a $2^{-j}$ neighborhood of $\xi_{t,0}(\mu)$.
\end{proof}

\section{Proof of theorem \ref{thm:keybilinear}}\label{sec:bilinear}
The bound \eqref{keybilinear} asserts that the phase space separation of $B_{\mu,\omega}^j(s)f$, $B_{\mu',\omega'}^j(s')g$ yields a gain of $(h2^{-2j})^{\frac{d+1}{q}-\frac{d-1}{2}}$ over what would be obtained by a formal application of the linear bounds \eqref{linearbounds}.  The only difficulty with linear bounds is that the symbols $a^j_{\mu,\omega}$ from \eqref{tmuomegafixeds} are not in $S_0(1)$ since there are losses when they are differentiated in $(x,\eta)$. However, we will see that the decay of the symbol in $x,\xi$ ensures that linear bounds for these symbols are valid nonetheless.  In particular, linear bounds will apply to the case $h \approx 2^{2j}$, where there insufficient phase space separation to apply bilinear bounds.

Our first task is to reduce matters to working with compactly supported symbols instead of the rapidly decaying symbols $a^j_{\mu,\omega}$.  This step is not necessary when $h \approx 2^{2j}$, so the discussion from here through Lemma \ref{thm:FIOPDOsym} will assume that $J+1 \leq j \leq 0$. Recalling from \eqref{bjseparation} that
\[
 d\big(\mathscr{D}_{\mu}^j,\mathscr{D}_{\mu'}^j\big)  \approx 2^j, \qquad J+1 \leq j \leq 0
\]
we may take smooth bump functions $c_\mu^j(x,\xi)$ and $c_{\mu'}^j(x,\xi)$ such that
\[
 d\left(\supp(1-c_\mu^j),\mathscr{D}_{\mu}^j\right) \geq \delta 2^{j}
\]
for some sufficiently small, but uniform constant $\delta >0$ while
\begin{equation}\label{cjseparation}
 d\big(\supp(c_\mu^j),\supp(c_{\mu'}^j)\big)\approx 2^j.
\end{equation}
Now extend the $c_\mu^j$  to a family $\{ c_\mu^j(s,\cdot) \}_{s \in [-\veps,\veps]}$ by defining
\[
 c_\mu^j(s,x,\xi) := c_\mu^j\left(\kappa_{0,s}(x,\xi)\right) = c_\mu^j\left((\kappa_{s,0})^{-1}(x,\xi)\right)
\]
Taking $\veps>0$ sufficiently small in \eqref{epsdef}, we may assume that for every $s \in [-\veps,\veps]$
\begin{align}
 d\big(\supp(1-c_\mu^j)(s,\cdot)),\kappa_{s,0}(\mathscr{D}_{\mu}^j) \big) \geq \delta 2^{j},\label{psibseparation}\\
 d\Big(\supp(c_\mu^j(s,\cdot)),\supp(c_{\mu'}^j(s,\cdot))\Big)\approx 2^j.\label{psimutualseparation}
\end{align}
Now let $C_\mu^j(s)$ denote the $h$-pseudodifferential operator with symbol $c_\mu^j(s,\cdot)$.

We will first show that linear estimates can be applied to
\begin{multline}\label{subtractmain}
h^{-\frac{2d}{q}+(d-1)}2^{-j(d-1-\frac{2(d+1)}{q})} \times \\
\left\|R\left( B_{\mu,\omega}^j(s)f,B_{\mu',\omega'}^j(s')g\right)- R\left((C^j_\mu(s)\circ B_{\mu,\omega}^j(s))f,(C^j_{\mu'}(s')\circ B_{\mu',\omega'}^j(s'))g\right)\right\|_{L^{\frac q2}}
\end{multline}
so that it is bounded by the last three terms on the right hand side of \eqref{keybilinear}.  In particular, this process will lead us to define
\begin{equation}\label{tildeBdef}
\tilde{B}_{\mu,\omega}^j(s) :=  (h2^{-2j})^{-\frac{d+1}{q}+\frac{d-1}{2}} (I-C_{\mu}^j(s))\circ B_{\mu,\omega}^j(s).
\end{equation}
After applying linear bounds, we will be reduced to
\begin{multline}\label{bilinearcj}
\left\| R\left((C^j_\mu(s)\circ B_{\mu,\omega}^j(s))f,(C^j_{\mu'}(s')\circ B_{\mu,\omega}^j(s'))g\right)\right\|_{L^{\frac q2}([-\veps,\veps]_t \times \RR^{d-1}_x)}\lesssim\\
h^{\frac{2d}{q}-(d-1)}2^{j(d-1-\frac{2(d+1)}{q})}
\|B_{\mu,\omega}^j(s)f\|_{L^2(\RR^{d-1})}\|B_{\mu',\omega'}^j(s')g\|_{L^2(\RR^{d-1})},
\end{multline}
which will follow from bilinear estimates of S. Lee \cite{leebilinear}, using the separation in \eqref{psimutualseparation}.

\begin{lemma}
The composition $(I-C_\mu^j(s))\circ B_{\mu,\omega}^j(s)$ is an $h$-PDO with whose symbol $q(s,x,\xi)$ satisfies
\[
|\prtl^\alpha_{x,\xi} q(s,x,\xi)| \lesssim_{\alpha,N}
(h2^{-2j}) h^{-\frac{|\alpha|}{2}} (1+ h^{-\frac 12}d(x,\xi;\kappa_{s,0}(\mathscr{D}_{\mu}^j)))^{-N}.
\]
\end{lemma}
The lemma thus ensures that $\tilde{B}_{\mu,\omega}^j(s)$ as defined in \eqref{tildeBdef} satisfies \eqref{actualdecay} since $h2^{-2j}\leq (h2^{-2j})^{\frac{d+1}{q}-\frac{d-1}{2}}$ for $q \geq 2$.
\begin{proof}
Observe that
\[
q(s,x,\xi)=
\frac{1}{(2\pi h)^{d-1}} \int_{\RR^{2(d-1)}}e^{\frac ih \langle x-y,\eta-\xi\rangle} (1-c_\mu^j)(s,x,\eta)b_{\mu,\omega}^j(s,y,\xi)\,dyd\eta,
\]
and consider more general oscillatory integrals of the form
\begin{equation}\label{compositionmodel}
\frac{1}{(2\pi h)^{d-1}} \int_{\RR^{2(d-1)}}e^{\frac ih \langle x-y,\eta-\xi\rangle} A(x,\eta,y,\xi)\,dyd\eta.
\end{equation}
where $\supp(A(\cdot,y,\xi)) \subset \supp((1-c_\mu^j)(s,\cdot))$ for every $(y,\xi)$ and
\[
|\prtl^\alpha A(x,\eta,y,\xi)| \lesssim_{\alpha,N}  h^{-\frac{|\alpha|}{2}} (1+ h^{-\frac 12}d(y,\xi;\kappa_{s,0}(\mathscr{D}_{\mu}^j)))^{-N}
\]
with implicit constants depending only on those defining $c_\mu^j$, $b_\mu^j$ and their derivatives. The main idea is that through induction, $(h^{\frac 12}\prtl_{x,\xi})^\alpha q(s,x,\xi)$ is an oscillatory integral of this type, which follows easily by integration by parts since
\[
\prtl_{x_i}e^{\frac ih \langle x-y,\eta-\xi\rangle} = -\prtl_{y_i} e^{\frac ih \langle x-y,\eta-\xi\rangle},
\qquad
\prtl_{\xi_i}e^{\frac ih \langle x-y,\eta-\xi\rangle} = -\prtl_{\eta_i} e^{\frac ih \langle x-y,\eta-\xi\rangle}.
\]
Hence we are left to show that
\begin{equation}\label{compositionmodelbound}
|\eqref{compositionmodel}| \lesssim h2^{-2j}(1+ h^{-\frac 12}d(x,\xi;\kappa_{s,0}(\mathscr{D}_{\mu}^j)))^{-N}.
\end{equation}
The differential operator
\[
\frac{1+(h^{-\frac 12}(\xi-\eta,x-y))\cdot(h^{\frac 12} D_{(y,\eta)})}{1+h^{-1}|(\xi-\eta,x-y)|^2}
\]
preserves the phase function in \eqref{compositionmodel}, so integration by parts shows that
\begin{multline*}
|\eqref{compositionmodel}| \lesssim_N \\
\frac{1}{(2\pi h)^{d-1}} \int (1+h^{-\frac 12}|(x-y,\xi-\eta)|)^{-3N}\left(1+ h^{-\frac 12}d\left(y,\xi;\kappa_{s,0}(\mathscr{D}_{\mu}^j)\right)\right)^{-2N} dyd\eta
\end{multline*}
and the domain of integration can be restricted to $\eta \in \supp((1-c_\mu^j)(s,x,\cdot))$.  The first factor in the integrand allows us to replace the pair $(y,\xi)$ in the second by either $(x,\eta)$ or $(x,\xi)$.  In particular, when $N \geq 2$, the integrand is dominated by
\begin{equation*}
  (1+h^{-\frac 12}|(x-y,\xi-\eta)|)^{-N}\big(h^{-\frac 12}d(x,\eta;\kappa_{s,0}(\mathscr{D}_{\mu}^j))\big)^{-2}
  \Big(1+ h^{-\frac 12}d\big(x,\xi;\kappa_{s,0}(\mathscr{D}_{\mu}^j)\big)\Big)^{-N}.
\end{equation*}
Since $(x,\eta) \in \supp((1-c_\mu^j)(s,\cdot))$, \eqref{psibseparation} shows that the second factor here is $O(h2^{-2j})$. Since the first factor means that integration in $y,\eta$ yields a gain of $O(h^{d-1})$, the bound \eqref{compositionmodelbound} follows.
\end{proof}

\begin{lemma}\label{thm:FIOPDOsym} Let $\tilde{c}_\mu^j(t,s,x,\xi) = c_\mu^j(s,d_\eta\phi(t,s,x,\xi),\xi)$
There exists symbols $q^j_{\mu,\omega}$ and $r^j_{\mu,\omega}$ such that
\[
T_{\mu,\omega}^{j,s}(C^j_\mu(s)f)=\frac{1}{(2\pi h)^{d-1}} \iint e^{\frac ih( \phi(t,s,x,\xi)-\langle y, \xi \rangle )} (q^j_{\mu,\omega}+r^j_{\mu,\omega})(t,s,x,\xi) f(y)\, dy d\xi ,
\]
with $|\prtl^\alpha_{x,\xi} q^j_{\mu,\omega}(t,s,\cdot)|\lesssim_\alpha 2^{-j|\alpha|}$ and
\begin{equation}\label{symbolsupp}
\supp(q^j_{\mu,\omega}(t,s,\cdot)) \subset \supp(\tilde{c}^j_\mu(t,s,\cdot)) = \overline{\{(x,\xi):c_\mu^j(s,d_\eta\phi(t,s,x,\xi),\xi) \neq 0 \}}
\end{equation}
while
\begin{equation}\label{rmuregularity}
|\prtl^\alpha_{x,\xi} r^j_{\mu,\omega}(t,s,\cdot)|
\lesssim_{\alpha,N}
(h2^{-2j})2^{-j|\alpha|}
\left(1+h^{-1}2^jd\left(x,\xi;\supp(\tilde{c}^j_\mu(t,s,\cdot))\right)\right)^{-N}.
\end{equation}
\end{lemma}
Again, since $h2^{-2j}\leq (h2^{-2j})^{\frac{d+1}{q}-\frac{d-1}{2}}$ the gain satisfied by $r^j_{\mu,\omega}$ will be sufficient.
\begin{proof}
The Schwartz kernel of the composition $T^{j,s}_{\mu,\omega}\circ C^j_\mu(s)$ is the oscillatory integral
\begin{equation}\label{FIOPDOsym}
\frac{1}{(2\pi h)^{d-1}}\iint e^{\frac ih(\phi(t,s,x,\eta)+\langle y,\xi-\eta\rangle)}a^j_{\mu,\omega}(t,s,x,\eta)c^j_\mu(t,s,y,\xi)\,dyd\eta .
\end{equation}
Its critical points satisfy $d_\eta\phi(t,s,x,\eta) = y$ and $\xi=\eta$.  The lemma thus follows from a variation on the stationary phase arguments in Theorem \ref{thm:egorov} (taking $M=1$ in the analog of \eqref{Jhuexpansion}).
\end{proof}

\begin{lemma}\label{lem:linearbounds}
Let $T_\mu^j f$ be an oscillatory integral operator defined by
\[
(T_\mu^j f)(t,x) = \frac{1}{(2\pi h)^{d-1}}\iint_{\RR^{2(d-1)}} e^{\frac ih (\phi(t,s,x,\xi)-\langle y,\xi \rangle)} a_\mu^j(t,s,x,\xi) f(y)\,dyd\xi ,
\]
where $J \leq j \leq 0$ and
\begin{equation}\label{ajdecay}
|\prtl^\alpha_{x,\xi} a_\mu^j(t,s,\cdot)| \lesssim_\alpha 2^{-j|\alpha|} (1+2^{-j}|x-x_{t,0}(\mu)|)^{-2d}(1+2^{-j}|\xi-\xi_{s,0}(\mu)|)^{-2d}.
\end{equation}
Then for some implicit constant depending only on the bounds for $a^j_\mu$, $\phi$ and their derivatives, we have
\[
\|T_\mu^j f\|_{L^q([-\veps,\veps]_t \times \RR^{d-1}_x)} \lesssim
h^{-\frac{d-1}{2}(\frac 12-\frac 1q)}\|f\|_{L^2(\RR^{d-1}_y)}.
\]
\end{lemma}

We postpone a thorough proof of these bounds to the end of the section, though as noted, they essentially are a consequence of the standard theory of Carleson-Sj\"olin integral operators.  For now we note that this will settle Theorem \ref{thm:keybilinear} when $j=J$ (strictly speaking this means that we take $\tilde{B}_{\mu,\omega}^J(s)=0$ as the enlarged family is not needed in this case).  Moreover, the linear estimates imply that \eqref{subtractmain} is bounded by the last three terms in \eqref{keybilinear}. Indeed, the lemma and the definition of $\tilde{B}_{\mu,\omega}^j(s)$ establishes that
\begin{multline*}
\left\|T^{j,s}_{\mu,\omega} \left((I-C_\mu^j(s))\circ B^j_{\mu,\omega}(s)f\right) \right\|_{L^q([-\veps,\veps]_t \times \RR^{d-1}_x)} \lesssim\\
h^{-\frac{d-1}{2}(\frac 12-\frac 1q)}(h2^{-2j})^{\frac{d+1}{q} - \frac{d-1}{2}}
\|\tilde{B}_{\mu,\omega}^j(s)f\|_{L^2(\RR^{d-1}_y)},
\end{multline*}
which yields the desired gain when paired with the linear bound on the second factor. Moreover, the gain of $h2^{-2j}\leq (h2^{-2j})^{\frac{d+1}{q}-\frac{d-1}{2}}$ in the symbol of $r^j_{\mu,\omega}$ shown in \eqref{rmuregularity} means that it suffices to prove \eqref{bilinearcj} with the compactly supported $q^j_{\mu,\omega}$ replacing $a^j_{\mu,\omega}$ in the definition of $T^{j,s}_{\mu,\omega}$.

\subsection{Bilinear estimates}\label{ss:bilinearestimates}
We now complete the proof of \eqref{bilinearcj}.  Set
\begin{equation}\label{uidef}
 u_1(\eta) = \frac{2^{j(d-1)}}{(2\pi h)^{(d-1)}}\mathscr{F}_h(B_{\mu,\omega}^j(s)f)(2^j\eta), \quad
 u_2(\eta) = \frac{2^{j(d-1)}}{(2\pi h)^{(d-1)}}\mathscr{F}_h(B_{\mu',\omega'}^j(s')g)(2^j\eta)
\end{equation}
where $\mathscr{F}_h$ denotes the semiclassical Fourier transform.  Observe that by the semiclassical Plancherel identity,
\begin{equation}\label{dilatedplancherel}
 \|u_1\|_{L^2(\RR^{d-1})} =
  \left(\frac{2^{j}}{2\pi h}\right)^{\frac{d-1}{2}}\|B_{\mu,\omega}^j(s)f\|_{L^2(\RR^{d-1})},
\end{equation}
and an analogous computation establishes the $L^2$ norm of $u_2$.

By the discussion following Lemma \ref{lem:linearbounds}, we may replace the symbol $a^j_{\mu,\omega}$ in the definition of $T^{j,s}_{\mu,\omega}$ by the compactly supported $q^j_{\mu,\omega}$ defined in Lemma \ref{thm:FIOPDOsym}. Next set
\[
 \tilde{\phi}(t,s,x,\eta) = 2^{-2j}\phi(t,s,2^jx,2^j\eta), \qquad \tilde{p}(t,x,\xi) = 2^{-2j}p(t,2^jx,2^j\xi),
\]
so that $\tilde{\phi}$ solves the eikonal equation
\begin{equation}\label{tildeeikonal}
 \prtl_t \tilde{\phi}(t,s,x,\eta) + \tilde{p}(t,x,d_x\tilde{\phi}(t,s,x,\eta))=0.
\end{equation}
Also, let $\tilde{\kappa}_{t,s}(x,\xi)$ denote the time $t$ value of the integral curve determined by the Hamiltonian vector field of $\tilde{p}(r,x,\xi)$ such that $\tilde{\kappa}_{r,s}(x,\xi)\big|_{r=s} = (x,\xi)$. It is verified that the $\tilde{\kappa}_{t,s}$ are related to the original mappings $\kappa_{t,s}$ by
\[
 \tilde{\kappa}_{t,s}(x,\xi) = 2^{-j}\kappa_{t,s}(2^j x,2^j\xi).
\]
Hence $\tilde{\phi}$ is a generating function for $\tilde{\kappa}_{t,s}$, satisfying relations analogous to \eqref{generatingfunction}.  The mixed Hessian $d_x d_\xi\tilde{\phi}=I+O(\veps)$ is nonsingular by the chain rule and \eqref{genfcnnondeg}.

Now make the following definitions
\[
\tilde{q}^j_{\mu,\omega}(t,s,x,\xi) := q^j_{\mu,\omega}(t,s,2^jx,2^j\xi), \qquad \tilde{q}^j_{\mu',\omega'}(t,s,x,\xi) := q^j_{\mu',\omega'}(t,s',2^jx,2^j\xi),
\]
\begin{align}
 T_1u_1(t,x) &:= \int e^{\frac{i2^{2j}}{h}\tilde{\phi}(t,s,x,\xi)}\tilde{q}^j_{\mu,\omega}(t,s,x,\xi)u_1(\xi)\,d\xi,\label{T1def}\\
 T_2u_2(t,x) &:= \int e^{\frac{i2^{2j}}{h}\tilde{\phi}(t,s',x,\eta)}\tilde{q}^j_{\mu',\omega'}(t,s',x,\eta)u_2(\eta)\,d\eta.\label{T2def}
\end{align}
so that after the symbol replacement above, we have that up to acceptable error,
\[
R\left((C^j_{\mu}(s)\circ B_{\mu,\omega}^j(s))f,(C^j_{\mu'}(s')\circ B_{\mu',\omega'}^j(s'))g\right)(t,2^jx) = T_1u_1(t,x)\cdot T_2u_2(t,x)
\]
The bound \eqref{bilinearcj} will thus follow by dilating variables $x \mapsto 2^jx$ and the bound
\begin{equation}\label{bilineardilated}
 \left\|T_1 u_1 T_2 u_2\right\|_{L^{\frac q2}([-\veps,\veps]_t\times \RR^{d-1}_x)} \lesssim (h2^{-2j})^{\frac{2d}{q}} \|u_1\|_{L^2(\RR^{d-1})} \|u_2\|_{L^2(\RR^{d-1})}.
\end{equation}

To show \eqref{bilineardilated}, we will use bilinear estimates.  When $d=2$, these are implicit in the work of H\"ormander \cite{HorOsc}  and when $d \geq 3$, they are a consequence of results of S. Lee \cite[Theorem 1.1]{leebilinear} and the epsilon removal lemma in \cite[Theorem 3.3]{blairsogge}.  To align the notation here with the latter works, define
\[
 \delta(t,x,\xi,\eta) = (d_\xi\tilde{p})(t,x,d_x\tilde{\phi}(t,s,x,\xi))-(d_\eta\tilde{p})(t,x,d_x\tilde{\phi}(t,s',x,\eta)).
\]
Since the Hessian $\prtl^2_{\eta_i\eta_j} \tilde{p}(t,x,\eta) = (\prtl^2_{\eta_i\eta_j} p)(t,2^j x,2^j \eta)$ is positive definite,
\[
| \delta(t,x,\xi,\eta)| \approx |d_x\tilde{\phi}(t,s,x,\xi)-d_x\tilde{\phi}(t,s',x,\eta)|.
\]
The crucial hypothesis to be verified to apply these bounds is that the following expression is uniformly bounded from below when $(2^jx,2^j\xi) \in \supp(q^j_{\mu}(s,\cdot))$ and $(2^jx,2^j\eta) \in \supp(q^j_{\mu'}(s',\cdot))$,
\[
 \left| \left\langle \left[ d_x d_\xi \tilde{\phi}\right] \delta(t,x,\xi,\eta),\left[ d_x d_\xi \tilde{\phi}\right]^{-1} \left[d_\xi d_\xi \tilde{p}(t,x,d_x\tilde{\phi}) \right]^{-1}\delta(t,x,\xi,\eta)\right\rangle \right|,
\]
where the derivatives of $d_x\tilde{\phi}$, $d_x d_\xi \tilde{\phi}$ are either both evaluated at $(t,s,x,\xi)$ or both evaluated at $(t,s',x,\eta)$. Given that $d_x d_\xi\tilde{\phi}=I + O(\veps)$, it suffices to see that
\begin{equation}\label{phiseparation}
 \left|d_x\tilde{\phi}(t,s,x,\xi)-d_x\tilde{\phi}(t,s',x,\eta) \right| \approx 1.
\end{equation}

To see \eqref{phiseparation}, note that by \eqref{symbolsupp}, $(x,\xi) \in \supp(\tilde{q}^j_\mu(s,\cdot))$ implies that
\[
\left(d_\xi\phi(t,s,2^j x, 2^j\xi),2^j\xi \right) \in \supp(c^j_\mu(s,\cdot)) = \kappa_{s,0}\left(\supp(c^j_\mu(0,\cdot))\right).
\]
Applying $\kappa_{t,s}$ to both sides of this, we see that by \eqref{generatingfunction} this is equivalent to
\[
\left(2^j x, d_x\phi(t,s,2^j x, 2^j\xi)\right) \in \kappa_{t,0}\left(\supp(c^j_\mu(0,\cdot))\right).
\]
The same reasoning shows that if $(2^j x,2^j \eta) \in \supp(q^j_{\mu'}(s',\cdot))$, then
\[
\left(2^j x, d_x\phi(t,s',2^j x, 2^j\eta)\right) \in \kappa_{t,0}\big(\supp(c^j_{\mu'}(0,\cdot))\big).
\]
By \eqref{cjseparation}, this gives
$
\left|d_x\phi(t,s,2^j x, 2^j\xi) -d_x\phi(t,s',2^j x, 2^j\eta)\right| \approx 2^j,
$
and \eqref{phiseparation} follows.

\subsection{Linear estimates}
Here we prove Lemma \ref{lem:linearbounds}, using the rescaling from the previous subsection \S\ref{ss:bilinearestimates}.  Define $\tilde{a}_{\tilde{\mu}}^j(t,s,x, \xi) :=a_\mu^j(t,s,2^j x,2^j \xi)$, let $u_1$ be as in \eqref{uidef} but replacing $B_{\mu,\omega}^j(s)f$ by $f$,  and let $T_1$ be the operator with Schwartz kernel given by the oscillatory integral
\[
\int e^{\frac{i2^{2j}}{h} (\tilde{\phi}(t,s,x,\xi)-\langle y,\xi \rangle)} \tilde{a}_{\tilde{\mu}}^j(t,s,x, \xi)\, d\xi .
\]
We thus have that $\tilde{a}_{\tilde{\mu}}^j(t,s,\cdot)\in S_0(1)$, with
\begin{equation}\label{ajtildedecay}
|\prtl^\alpha_{x,\xi}\tilde{a}_{\tilde{\mu}}^j(t,s,\cdot)| \lesssim_\alpha
\langle x-\tilde{x}_{t,0}(\tilde{\mu})\rangle^{-2d}\langle \xi-\tilde{\xi}_{s,0}(\tilde{\mu})\rangle^{-2d}, \quad \tilde \mu = 2^{-j}\mu
\end{equation}
writing $\tilde{\kappa}_{r,0}(\tilde{\mu})$ in terms of components $(\tilde{x}_{r,0}(\tilde{\mu}),\tilde{\xi}_{r,0}(\tilde{\mu}))$.  It now suffices to show
\begin{equation}\label{dilatedlinearbound}
\|T_1 u_1\|_{L^q([-\veps,\veps]_t \times \RR^{d-1}_x)} \lesssim
(h2^{-2j})^{\frac{d-1}{2}(\frac 12+\frac 1q)}\|u_1\|_{L^2(\RR^{d-1}_y)}.
\end{equation}
Indeed, given \eqref{dilatedplancherel}, the right hand side is $\approx h^{-\frac{d-1}{2}(\frac 12-\frac 1q)}2^{-\frac{j(d-1)}{q}}\|f\|_{L^2}$, and the loss in $2^j$ cancelled by reverting back to the original coordinates on the left hand side.  This is the same gain observed in the first inequality in \eqref{linearbounds}, but with frequency $h^{-1}2^{2j}$.  The main observation is that $T_1$ is an oscillatory integral operator with a Carleson-Sj\"olin phase as defined in \cite[\S2.2]{soggefica}.   Indeed,
\begin{equation}\label{tildenonsing}
d_x d_\xi \tilde{\phi}(t,s,x,\xi) =(d_x d_\xi \phi)(t,s,2^j x,2^j \xi),
\end{equation}
and the latter defines a nonsingular $(d-1)\times (d-1)$ matrix.  Recalling \eqref{tildeeikonal}, it suffices to check that for $(t,x)\in \supp(\tilde{q}^j_\mu(\cdot,s,\cdot,\xi))$ fixed
\begin{equation}\label{ptildegraph}
\left\{(-\tilde{p}(t,x,d_x\tilde{\phi}(t,s,x,\xi)), d_x\tilde{\phi}(t,s,x,\xi)): \xi \in \supp(\tilde{q}^j_\mu(t,s,x,\cdot)) \right\}
\end{equation}
is an embedded hypersurface in $T_{(t,x)}^*\RR^{d}$ with nonvanishing principal curvatures.  But this follows from the fact that this is a local reparameterization of the graph of $\eta\mapsto -\tilde{p}(t,x,\eta)$ and $\prtl^2_{\eta_i\eta_j} \tilde{p}(t,x,\eta) = (\prtl^2_{\eta_i\eta_j} p)(t,2^j x,2^j \eta)$.

At this point, \eqref{dilatedlinearbound} formally follows from the standard theory of Carleson-Sj\"olin integrals due to H\"ormander and Stein (see e.g. \cite[Theorem 2.2.1]{soggefica}).  However, care must be taken here since the amplitude $\tilde{a}_{\tilde{\mu}}^j$ is not compactly supported in a set of uniform diameter. However, it does decay rapidly outside such a set, which is enough.  Seeing this is a matter of tracing through the proof in \cite[Theorem 2.2.1(1)]{soggefica}, which we do here.  We first observe that given \eqref{tildenonsing}, the bound \eqref{dilatedlinearbound} with $q=2$ is a consequence of H\"ormander's $L^2$ theorem \cite{HorOsc} and the decay in \eqref{ajtildedecay}.  Indeed, this decay allows us to write the amplitude as a sum
\[
\tilde{a}_{\tilde{\mu}}^j(t,s,x,\xi) = \sum_{\tilde{\nu} =(\tilde{\nu}_x,\tilde{\nu}_\xi)\in \mathbb{Z}^{2(d-1)}} \big\langle \tilde{\nu}_x-\tilde{x}_{t,0}(\tilde{\mu}) \big\rangle^{-2d} \big\langle \tilde{\nu}_\xi-\tilde{\xi}_{s,0}(\tilde{\mu}) \big\rangle^{-2d} \tilde{a}_{\tilde{\mu},\tilde{\nu}}^j(t,s,x,\xi),
\]
and we may take $\tilde{a}_{\tilde{\mu},\tilde{\nu}}^j(t,s,\cdot) \in S_0(1)$ with
$\supp(\tilde{a}_{\tilde{\mu},\tilde{\nu}}^j(t,s,\cdot)) \subset \{|(x,\xi)-\tilde{\nu}| \lesssim 1\}$.
For fixed $t$, H\"ormander's theorem can now be applied in $L^2(\RR^{d-1})$ to each operator determined by the decomposition here, and the decay in $\tilde{\nu}$ ensures that we may sum over the contribution of each of these operators without loss.  This results in
\begin{equation*}
  \|T_1\|_{L^2(\RR^{d-1}) \to L^{2}([-\veps,\veps] \times \RR^{d-1})} = O((h2^{-2j})^{\frac{d-1}{2}}).
\end{equation*}

It now suffices to prove \eqref{dilatedlinearbound} with $q=\frac{2(d+1)}{d-1}$, that is,
\begin{equation*}
  \|T_1\|_{L^2(\RR^{d-1}) \to L^{\frac{2(d+1)}{d-1}}([-\veps,\veps] \times \RR^{d-1})} = O((h2^{-2j})^{\frac{d(d-1)}{2(d+1)}}).
\end{equation*}
Defining the fixed $t$ operator $(T_1^tu_1)(x) = (T_1u_1)(t,x)$, a duality and fractional integration argument along with the aforementioned $L^2$ bounds means it suffices to show that
\[
\|T_1^t(T_1^{\tilde{t}})^*\|_{L^1(\RR^{d-1}) \to L^\infty(\RR^{d-1})} \lesssim (h2^{-2j}|t-\tilde{t}|)^{-\frac{d-1}{2}}
\]
which is a consequence of
\[
\left|\int e^{\frac{i2^{2j}}{h}(\tilde{\phi}(t,s,x,\xi)-\tilde{\phi}(\tilde{t},s,\tilde{x},\xi))}\tilde{a}_{\tilde{\mu}}^j(t,s,x,\xi)
\overline{\tilde{a}^j_{\tilde{\mu}}(\tilde t,s,\tilde{x},\xi)}\,d\xi\right| \lesssim (h2^{-2j}|t-\tilde{t}|)^{-\frac{d-1}{2}}.
\]
As in \cite[\S2.2]{soggefica}, this in turn follows from the fact that \eqref{ptildegraph} ensures that any critical points in the phase are nondegenerate.  As before, this requires some care since the integration is not over a set of uniform size, but this can be overcome by using a partition of unity adapted to cubes of sidelength 1 in $\xi$, then using the decay in $\xi$ away from $\tilde{\xi}_{s,0}(\tilde{\mu})$.


\begin{thebibliography}{99}
\bibitem{abrahammarsden} Abraham, R; Marsden, J.E. {\em Foundations of Mechanics}, second edition. Benjamin/Cummings Publishing Co., Reading, Mass., 1978.
\bibitem{Berard}  B\'erard, P., {\em On the wave equation on a compact Riemannian manifold without conjugate points}, Math. Z.  {\bf 155} (1977), no. 3,  249-276.
\bibitem{blairsogge} Blair, M. D.; Sogge, C. D. {\em On Kakeya-Nikodym Averages, $L^p$-norms and lower bounds for nodal sets of eigenfunctions in higher dimensions}. J. Eur. Math. Soc. \textbf{17} (2015) 2513-2543.
\bibitem{blairsoggerefined} Blair, M. D.; Sogge, C. D. {\em Refined and Microlocal Kakeya-Nikodym Bounds of Eigenfunctions in two dimensions.} Anal. PDE \textbf{8} (2015), 747-764.

\bibitem{blairsoggetop} Blair, M. D.; Sogge, C. D. {\em Toponogov's theorem and improved estimates of eigenfunctions.} To appear, J. Differential Geom.
\bibitem{bourgain} Bourgain, J. \emph{Geodesic restrictions and $L^p$-estimates for eigenfunctions of Riemannian surfaces}. Linear and complex analysis, Amer. Math. Soc. Transl. {\bf 226} (2009),  Amer. Math. Soc., Providence, RI,  27-35.

\bibitem{christiansonmonodromy} Christianson, H. \emph{Quantum monodromy and nonconcentration near a closed semi-hyperbolic orbit}. Trans. Amer. Math Soc. {\bf 363} (2011), no. 7, 3373-3438.

\bibitem{HT} Hassell, A; Tacy, M. {\em  Improvement of eigenfunction estimates on manifolds of nonpositive curvature}.  Forum Math \textbf{27} (2015), no. 3, 1435-1451.

\bibitem{HR} Hezari, H; Rivi\`ere, G. {\em $L^p$ norms, nodal sets, and quantum ergodicity}. Adv. Math \textbf{290} (2016), 938-966.

\bibitem{HorOsc} H\"ormander, L. \emph{Oscillatory integrals and multipliers on $FL^p$}. Ark. Math. {\bf 11} (1973), 1-11.

\bibitem{ktzsemi} Koch, H.; Tataru, D.; Zworski, M. \emph{Semiclassical $L^p$ estimates}. Annales Henri Poincar\'e {\bf 8} (2007), 885-916.
\bibitem{leebilinear} Lee, S. \emph{Linear and bilinear estimates for oscillatory integral operators related to restriction to hypersurfaces}.   J. Funct. Anal. {\bf 241} (2006),  56--98.

\bibitem{sogge88} Sogge, C.D. \emph{Concerning the $L^p$ norm of spectral clusters for second-order elliptic operators on compact manifolds.} J. Funct. Anal. {\bf 77}  (1988),  123--138.
\bibitem{soggefica} Sogge, C.D. \emph{Fourier integrals in classical analysis}. Cambridge Tracts in Math., Cambridge Univ. Press, Cambridge, 1993.

\bibitem{soggeforthcoming} Sogge, C.D.  \emph{Problems related to the concentration of eigenfunctions}. Preprint.
\bibitem{soggekaknik} Sogge, C.D.  \emph{Kakeya-Nikodym averages and $L^p$-norms of eigenfunctions}. Tohoku Math. J. {\bf 63} (2011),  519-538.
\bibitem{soggetothzelditch} Sogge, C.D.; Toth, J.; Zelditch, S. {\em About the blowup of quasimodes on Riemannian manifolds.} J. Geom. Anal. \textbf{2011}, 21 (1), 150-173.


\bibitem{soggezelduke} Sogge, C.D.; Zelditch, S. \emph{Riemannian manifolds with maximal eigenfunction growth}. Duke Math. J. {\bf 114} (2002), 387--437.

\bibitem{soggezelfocal} Sogge, C.D.; Zelditch, S. \emph{Focal points and sup-norms of eigenfunctions}. Rev. Mat. Iberoam. {\bf 32} (2016), no. 3, 971-994.

\bibitem{taobilinear} Tao, T. \emph{A sharp bilinear restriction estimates for paraboloids}.  Geom. Funct. Anal. {\bf 13}  (2003), 1359-1384.
\bibitem{tvv} Tao, T.; Vargas, A.; Vega, L. \emph{A bilinear approach to the restriction and Kakeya conjectures}. J. Amer. Math. Soc. {\bf 11}  (1998), 967--1000.

\bibitem{taylorpdo} Taylor, M. E. \emph{Pseudodifferential operators}. Princeton Mathematical Series, Princeton Univ. Press, Princeton, NJ, 1981.


\bibitem{wolff} Wolff, T. \emph{A sharp cone restriction estimate}.  Ann. of Math. {\bf 153}  (2001), 661--698.

\bibitem{zworski} Zworski, M. \emph{Semiclassical Analysis}. Graduate Studies in Mathematics {\bf 138}, Amer. Math. Soc., Providence, 2012.


\end{thebibliography}
\end{document}